\documentclass[11pt]{article}

\usepackage{amssymb}
\usepackage{latexsym}
\usepackage{graphicx}
\usepackage{amsmath}
\usepackage{amsthm}
\usepackage{amsfonts}
\usepackage{amscd}
\usepackage{bm}
\theoremstyle{plain}
\newtheorem{theorem}{Theorem}[section]

\newtheorem{lemma}[theorem]{Lemma}
\newtheorem{definition}[theorem]{Definition}
\newtheorem{corollary}[theorem]{Corollary}

\newtheorem{example}[theorem]{Example}
\newtheorem{remark}[theorem]{Remark}

\newtheorem{fact}[theorem]{Fact}

\title{Perfect Codes for Generalized Deletions from Minuscule Elements of Weyl Groups}
\author{Manabu Hagiwara\\Department of Mathematics and Informatics\\Chiba University}
\date{}

\begin{document}

\maketitle

\allowdisplaybreaks[4]
\section{Introduction}
This paper discusses a connection between insertion/deletion (ID) codes and minuscule elements of Weyl groups.

ID codes are a class of error-correcting codes for insertion errors and/or deletion errors in coding theory.
The concept of ID has been proposed in 1960's by Levenshtein {levenshtein1966binary}.
He found that Varshamov and Tenengolts (VT) codes \cite{varshamov1965code} are applicable to correct a single insertion or deletion error.
An insertion error for a sequence is a transformation that inserts certain symbols in the original sequence,
hence it increases the length of the sequence, e.g.) boy $\rightarrow$ buoy.
A deletion error is a sort of an opposite operation against an insertion error.
It deletes certain symbols from the original sequence,
hence it decrease the length.
After this kind of errors has been introduced as a model of synchronization error in communication scenario,
it is applied to DNA analysis \cite{xu2005survey,ossowski2008sequencing,kurtz2001reputer},
race-track memory error-correction \cite{vahid2017correcting,chee2017codes},
natural language processing \cite{och2003minimum,brill2000improved}
and so on.

Minuscule elements are elements of Weyl groups \cite{stembridge2001minuscule}
and appear to representation theory \cite{green2013combinatorics}
and algebraic combinatorics \cite{green2002321}
Reduced expressions of minuscule elements are related to 
fully commutative elements \cite{stembridge1996fully,biagioli2013fully,jouhet2014long},
and generalized Young diagrams,
e.g. 
d-complete posets \cite{proctor1999dynkin},
minuscule posets \cite{proctor1999minuscule},
minuscule heaps \cite{hagiwara2004minuscule} and so on.

ID-codes and minuscule elements have been introduced
and been studied dependently.
Therefore any connection between two topics have not reported yet.
This paper discusses ideas to connect these topics.
Briefly speaking,
this paper is organized as follows:
In \S\ref{sec:BitSeqMinElemTypeB}, \ref{sec:MomentLevLeng}, \ref{sec:inserstionForM}, and \ref{sec:delAndPerfect},
relations between ``ID-codes and (standard) insertion/deletion operations'' and
``minuscule elements and other elements of Weyl groups of type B are explained.
In \S\ref{consTypeA}, \ref{sec:perfectTypeA}, and \ref{sec:BAIBAD},
generalizations of ID-codes and insertion/deletion operations are
proposed by replacing the type B of Weyl groups with the type A.
In \S\ref{sec:SimCard} and \ref{sec:SimInsSphere},
a certain similarity of properties between generalized notions
and standard operations are introduced.

This paper assumed readers' knowledge of Weyl groups,
e.g.) a book \cite{humphreys1992reflection}.
It is preferable to be able to read a paper \cite{stembridge2001minuscule}.
A part of this paper \S\ref{sec:BAIBAD} was presented a symposium ISIT(The IEEE International Symposium on Information Theory) 2017
\cite{hagiwara2017perfect}.

\section{Bit Sequences and Minuscule Elements of Type $B$}\label{sec:BitSeqMinElemTypeB}

A bit sequence is a sequence over a binary set $\{ 0, 1 \}$.
If a bit sequence belongs to $\{ 0, 1 \}^n$ for some non-negative integer $n$,
the sequence is called of length $n$.
In stead of a notation $\mathbf{x} := (x_1, x_2, \dots, x_n) \in \{ 0, 1 \}^n$,
we may allow to denote $\mathbf{x}$ by $x_1 x_2 \dots x_n$ in the paper.
For example, $01001$ is a bit sequence of length $5$.
Throughout this paper, bold letters denote bit-sequences,
i.e., $\mathbf{x}, \mathbf{y}$ and $\mathbf{z}$.

In this section, let us construct bijections between bit sequences,
a set of orbits of a vector by a Weyl group, minuscule elements
and a set of right cosets defined below.
For defining minuscule elements, we start with a definition of root system of type $B$.
For $n \ge 2$, $\Pi ( B_n )$ denotes the following set, as a subset of $\mathbb{R}^n$ with
the standard basis $\langle \epsilon_1, \epsilon_2, \dots, \epsilon_n \rangle$,
$$
\Pi ( B_n ) := \{ \epsilon_1 \} \cup \{ \epsilon_2 - \epsilon_1, \epsilon_3 - \epsilon_2, 
\dots, \epsilon_n - \epsilon_{n-1} \}.
\footnote{
A standard choice of a simple root system of type $B_n$ is
$\{ \epsilon_1 - \epsilon_2, \epsilon_2 - \epsilon_3, 
\dots, \epsilon_{n-1} - \epsilon_{n} \} \cup \{ \epsilon_n \}$.
However, 
the choice $\Pi ( B_n )$ is suitable for connecting
to insertion/deletion codes.
}
$$
The set $\Pi (B_n)$ is known as a simple root system of type $B_n$.
For simplicity, set $\alpha_1 := \epsilon_1$ and $\alpha_i := \epsilon_{i} - \epsilon_{i-1}$ for $2 \le i \le n$.
Hence $\Pi( B_n) = \{ \alpha_1, \alpha_2, \dots, \alpha_{n} \}.$

For $\Pi (B_n)$, the associated root system $\Phi( B_n )$ is
$$
\Phi( B_n ) = 
\{ \pm \epsilon_i \mid 1 \le i \le n \}
\cup
\{ \epsilon_j \pm \epsilon_i, - \epsilon_j \pm \epsilon_i \mid 1 \le i < j \le n \}.
$$
Similarly, the positive system $\Phi( B_n )^+$ is
$$
\Phi( B_n )^+ = 
\{ \epsilon_i \mid 1 \le i \le n \}
\cup
\{ \epsilon_j \pm \epsilon_i \mid 1 \le i < j \le n \}.
$$

A Weyl group $W( B_n )$ of type $B_n$ is defined as a group
generated by reflections $s_i : \mathbb{R}^n \rightarrow \mathbb{R}^n$
associated to $\alpha_i$, i.e.,
$$
s_i (v) := v - 2 \frac{(\alpha_i, v)}{(\alpha_i, \alpha_i)} \alpha_i ,
$$
where $(, )$ is the standard inner product of $\mathbb{R}^n$.
An element $w \in W( B_n )$ is called a $\lambda$-minuscule for $\lambda \in \mathbb{R}^n$,
if
there exists a reduced expression $s_{i_1} s_{i_2} \dots s_{i_r}$ of $w$ such that
$$
s_{i_j} s_{i_{j+1}} \dots s_{i_r} \lambda
=
\lambda
- \alpha_{i_r} - \dots - \alpha_{i_{j+1}} - \alpha_{i_j},
$$
for any $1 \le j \le r$. \cite{stembridge2001minuscule}

\begin{example}\label{constructionOfMinusculeElementsOfTypeB}
Set $\lambda_B := (1/2, 1/2, \dots, 1/2) \in \mathbb{R}^n$.
For a subset $J := \{ j_1, j_2, \dots, j_t \}$ of $\{1, 2, \dots, n \}$,
where $j_1 < j_2 < \dots < j_t$,
let us define an element $w_J \in W(B_n)$ as follows:
$$
w_J := w_{j_t} \dots w_{j_2} w_{j_1},
$$
where
$$ w_{j_i} := s_{j_i } s_{j_i - 1} \dots s_{1} $$
for $1 \le i \le t$.
Note that if $J$ is empty, let us define $w_J$ as the identity element.
Indeed, a reduced expression of $w_J$ is obtained if we replace
$w_{j_i}$ with $s_{j_i } s_{j_i - 1} \dots s_{1} $ from 
$w_J = w_{j_t} \dots w_{j_2} w_{j_1}$. (See \cite{stembridge2001minuscule}.)

Then $w_J$ is a $\lambda_B$-minuscule element.
The $j+1$th entry of a vector $w_J \lambda_B$ is
$- 1/2$ if $j \in J$ and $1/2$ otherwise.
Hence $2^n$ $\lambda_B$-minuscule elements are obtained by the construction above.
\end{example}

Let us denote the set of $\lambda_B$-minuscule elements constructed from Example \ref{constructionOfMinusculeElementsOfTypeB}
by $\mathcal{M}_n$.
It will be explained soon that any $\lambda_B$-minuscule element belongs to $\mathcal{M}_n$.
(See Remark \ref{remark:minusculeAndRep}.)

Let us consider a maximal parabolic subgroup of $W(B_n)$ that is generated by $\{  s_i \mid 1 \le i \le n-1 \}$
and denote the subgroup by $W( A_{n-1} )$.
Then $W(A_{n-1})$ is a Weyl group of type $A_{n-1}$ and hence its cardinality is $n!$.
Therefore the cardinality of $W( B_n ) / W( A_{n-1} )$ is 
$$
\# W( B_n ) / W( A_{n-1} ) = \frac{\# W( B_n ) }{\#  W( A_{n-1} )} = \frac{ 2^n n!}{ n!} = 2^n.
$$

\begin{remark}\label{remark:minusculeAndRep}
Since $W( A_{n-1} )$ is the stabilizer subgroup of $W( B_n )$ for $\lambda_B$,
any $\lambda_B$-minuscule elements is a minimal coset representative of a right coset $W(B_n) / W(A_{n-1})$.
In particular, $\mathcal{M}_n$ is a subset of the set of minimal coset representatives.
By comparing the cardinalities, $\mathcal{M}_n$ is the set of minimal coset representatives.
\end{remark}

We have seen that the following sets are of cardinality $2^n$:
\begin{itemize}
\item $\{ 0, 1\}^n$.
\item $W(B_n) \lambda_B = \{ (\lambda_1, \lambda_2, \dots, \lambda_n) \mid \lambda_i = 1/2 \text{ or } -1/2, 1 \le i \le n \}$.
\item $\mathcal{M}_n$.
\item $W( B_n ) / W( A_{n-1} )$.
\end{itemize}
Remark that the cardinality of the set of subsets of $\{1,2, \dots, n\}$ is also $2^n$.

Here four bijections whose domain is the power set $\mathcal{J}_n$ of a set $\{1,2, \dots, n \}$ are
introduced as follows:
\begin{itemize}
\item $f_{01} : \mathcal{J}_n \to \{ 0, 1 \}^n$,
$$ f_{01} ( J ) := x_1 x_2 \dots x_n,$$
where $x_i := 1$ if $i \in J$ and $x_i := 0$ otherwise.
\item $f_{1/2} : \mathcal{J}_n \to \{ 1/2, -1/2 \}^n$,
$$ f_{1/2} ( J ) := \lambda_1 \lambda_2 \dots \lambda_n,$$
where $\lambda_i := -1/2$ if $i \in J$ and $\lambda_i := 1/2$ otherwise.
\item $f_{M} : \mathcal{J}_n \to \mathcal{M}_n$,
$$ f_{M} ( J ) := w_J,$$
where $w_J$ is defined in Example \ref{constructionOfMinusculeElementsOfTypeB}.
\item $f_{/} : \mathcal{J}_n \to W(B_n) / W(A_{n-1})$,
$$ f_{/} ( J ) := [ w_J ].$$
\end{itemize}

\begin{lemma}\label{lemma:bit_bijections}
\begin{itemize}
\item[(i)] $f_{1/2} \circ f_{01}^{-1} (\mathbf{x}) = \lambda_B - \mathbf{x}.$
\item[(ii)] $f_{M} \circ f_{01}^{-1} (\mathbf{x}) = w_{\mathrm{supp}(\mathbf{x})},$
where $\mathrm{supp}(\mathbf{x}) = \{i \mid x_i \neq 0 \}.$
\item[(iii)] $f_{01} \circ f_{1/2}^{-1} (\lambda) = \lambda_B - \lambda.$
\item[(iv)] $f_{M} \circ f_{1/2}^{-1} (\lambda) = w_{\mathrm{neg}(\lambda)},$
where $\mathrm{neg}(\lambda) = \{i \mid \lambda_i <  0 \}.$
\item[(v)] $f_{/} \circ f_{M}^{-1} (w) = [w]$.
\item[(vi)] $f_{1/2} \circ f_{M}^{-1} (w) = w \lambda_B$.
\item[(vii)] $f_{1/2} \circ f_{/}^{-1} ([w]) = w \lambda_B$.
\item[(viii)] $f_{M} \circ f_{/}^{-1} ([w]) =$ the minimal coset representative of $[w]$.
\item[(ix)] $f_{01} \circ f_{M}^{-1} (w) = \lambda_B - w \lambda_B.$
\end{itemize}
\end{lemma}

\begin{proof}
(i) follows from $1/2 - 0 = 1/2$ and $1/2 - 1 = -1/2$.
Similarly (iii) follows from $1/2 - 1/2 = 0$ and $1/2 - (-1/2) = 1$.

For (ii), $f_{M} \circ f_{01}^{-1} ( \mathbf{x} ) = f_M (\mathrm{supp}( \mathbf{x} )) = w_{\mathrm{supp}( \mathbf{x}) }.$
Similarly, for (iii),
$f_{M} \circ f_{1/2}^{-1} ( \lambda ) = f_M (\mathrm{neg}( \lambda ) = w_{\mathrm{neg}( \lambda) }.$

For (v),
remember that there uniquely exists $J \in \mathcal{J}$ such that $w = w_J$.
Hence $f_{/} \circ f_{M}^{-1} (w) = f_{/} \circ f_{M}^{-1} (w_J) = f_{/} (J) = [w_J] = [w].$

For (vi),
we have $f_{1/2} \circ f_{M}^{-1} (w) = f_{1/2} \circ f_M ^{-1} (w_J) = f_{1/2}(J)$.
On the other hand,
set $\lambda := w_J \lambda_B$.
By the definition of $w_J$, $\lambda_i = -1/2$ if $i \in J$ and $\lambda_i = 1/2$ otherwise.
It implies $f_{1/2} (J) = \lambda$ and hence $f_{1/2} \circ f_{M}^{-1} (w) = w_J \lambda_B.$

For (vii), remember that there uniquely exist $w_J \in \mathcal{M}_n$ and $v \in W(A_{n-1})$ such that
$w = w_J v$.
Then $f_{1/2} \circ f_{/}^{-1} ([w])
= f_{1/2} \circ f_{/}^{-1} ([w_J])
= f_{1/2} ( J  )
= f_{1/2} \circ f_{M}^{-1} ( w_J ).$
On the other hand,
$w \lambda_B = w_J v \lambda_B = w_J \lambda_B$,
since $W(A_{n-1})$ is the stabilizer group for $\lambda_B$.
By (vii), $f_{1/2} \circ f_{M}^{-1} ( w_J ) = w_J \lambda_B$.
Hence $f_{1/2} \circ f_{/}^{-1} ([w]) = w \lambda_B$.

For (viii),
let $w_J$ be the minimal coset representative of $[w]$.
Since $[w] = [w_J]$,
$f_{M} \circ f_{/}^{-1} ( [w] )
= f_{M} \circ f_{/}^{-1} ( [w_J] )
= f_{M} ( J )
= w_J.$

For (ix),
$f_{01} \circ f_{M}^{-1} (w)
= (f_{01} \circ f_{1/2}^{-1}) \circ (f_{1/2} \circ f_{M}^{-1}) (w)
= (f_{01} \circ f_{1/2}^{-1}) (w \lambda_B)
 = \lambda_B - w \lambda_B.$
\end{proof}

\section{A Moment Function of Levenshtein Codes and a Length Function}\label{sec:MomentLevLeng}
\begin{definition}[Levenshtein Codes $\mathbb{L}_{n, a}$ and a Moment Function $\rho$]
For a positive integer $n$ and an integer $a$,
a set $\mathbb{L}_{n, a}$ is defined as
$$
\mathbb{L}_{n, a} := \{ \mathbf{x} \in \{ 0, 1 \}^n \mid \rho( \mathbf{x} ) \equiv a \pmod{n+1} \},
$$
where 
$\rho( x_1, x_2, \dots, x_n) := x_1 + 2 x_2 + \dots + n x_n.$
$\mathbb{L}_{n, a}$ is called a Levenshtein code \cite{levenshtein1966binary}
and $\rho$ a moment function.
\end{definition}

\begin{remark}
In coding theory, $\mathbb{L}_{n, a}$ is invented by Varshamov and Tenengolts
for correcting asymmetric error, that is a noise over a Z-channel \cite{varshamov1965code}.
Hence this code is called a VT code in some literatures \cite{bibak2018weight}
or a Varshamov code \cite{stanley1972study}.
However the code has not been known that it is available to correct a single insertion/deletion.
The first person who discovered insertion/deletion error-correction property is Levenshtein.
Therefore the code is called a Levenshtein code in this paper
since the main research interests are relations between insertion/deletion and Weyl groups.
\end{remark}

\begin{definition}[Length Function $l$]
For $w \in W(B_n)$,
$l(w)$ is defined as
$$
l(w) := \min \{ r \mid w = s_{i1} s_{i2} \dots s_{ir} \}.
$$
The function $l$ is called a length function.
\end{definition}

\begin{theorem}\label{theorem:relationMomentAndLength}
For any $w \in \mathcal{M}_n$,
$$
l(w) = \rho( \mathbf{x}_w ),
$$
where $\mathbf{x}_w := \lambda_B - w \lambda_B (= f_{01} \circ f_M ^{-1} (w))$,
and $l$ is the length function associated with $s_0, s_1, \dots, s_{n-1}$.
In other words, for any $\mathbf{x} \in \{0, 1 \}^n$, 
$$
\rho ( \mathbf{x} ) = l ( w_{ \mathrm{supp}( \mathbf{x} ) } ),
$$
where $\mathrm{supp}( \mathbf{x} ) = \{ i \mid x_i = 1 \} (= f_{M} \circ f_{01} ^{-1} (\mathbf{x}))$.
\end{theorem}
\begin{proof}
Let $s_{i_1} s_{i_2} \dots s_{i_r }$ be a reduced expression of $w$.
\begin{align*}
l(w) &= r \\
     &= ( \sum_{\alpha \in \Phi^+ }\alpha^{\vee} , \alpha_{i_1} + \alpha_{i_2} + \dots + \alpha_{i_r} )
        & \left( \sum_{\alpha \in \Phi^+ }\alpha^{\vee} , \alpha_i ) = 1 \right) \\
     &= ( \sum_{\alpha \in \Phi^+ }\alpha^{\vee} , \lambda_{B} - w \lambda_B )
        & \left( w \lambda_B = \lambda_B - \alpha_{i_1} - \alpha_{i_2} - \dots - \alpha_{i_r} \right) \\
     &= ( \epsilon_1 + 2 \epsilon_2 + \dots + n \epsilon_n , \lambda_B - w \lambda_B )
        & \left( \sum_{\alpha \in \Phi^+ }\alpha^{\vee} = \epsilon_1 + 2 \epsilon_2 + \dots + n \epsilon_n \right) \\
     &= ( \epsilon_1 + 2 \epsilon_2 + \dots + n \epsilon_n , \mathbf{x}_w )\\
     &= x_1 + 2 x_{2} + \dots + n x_{n}
        & \left( \text{by setting } x_1 x_2 \dots x_n := \mathbf{x} \right) \\
     &= \rho( \mathbf{x}_w ) 
\end{align*}
\end{proof}

Theorem \ref{theorem:relationMomentAndLength} enables us to define Levenshtein codes
in terms of Weyl group:
$$
\mathcal{M}_{n, a} := \{ w \in \mathcal{M}_n \mid l(w ) \equiv a \pmod{n+1} \}
$$
Therefore, we obtain corollaries of statements which are described by binary sequences and the moment function.

\begin{example}
Let us observe a generator polynomial of $\rho$, equivalently $l$, from both sides of
bit sequences and minuscule elements.

From the side of bit sequences,
we can directly calculate:
\begin{align*}
\sum_{x \in \{0, 1\}^n} X^{\rho(x) }
 &=
  \sum_{x_1 \in \{0, 1\} }X^{x_1} 
  \sum_{x_2 \in \{0, 1\} } X^{2 x_2}
  \dots
  \sum_{x_n \in \{0, 1\} } X^{n x_n}\\
 &=
  (1+X)(1+X^2) \dots (1+X^n)\\
 &= \prod_{1 \le i \le n} (1+X^i).
\end{align*}
On the other hand, from the side of minuscule elements,
we have
\begin{align*}
\sum_{w \in \mathcal{M}_n} X^{l (w) }
 &=
\sum_{w \in W(B_n)/W(A_{n-1}) : \text{ minimal coset rep.}} X^{l (w) }\\
 &=
 \frac{
  \sum_{w \in W(B_n)} X^{l (w) }
 }{
  \sum_{w \in W(A_{n-1}) } X^{l (w) } 
 }\\
  &= \frac{
    \prod_{1 \le i \le n} [2i]
     }{
    \prod_{2 \le i \le n} [i]
     }\\
  &=
    \prod_{1 \le i \le n}
     \frac{
      1 - X^{2i}
     }{
      1 - X^i
     }\\
  &=
    \prod_{1 \le i \le n}
      (1 + X^{i}).
\end{align*}
\end{example}

\section{Insertions and Insertion Spheres}\label{sec:inserstionForM}
For a binary sequence $\mathbf{x} := x_1 x_2 \dots x_n$,
an operation $I_{i, b}^{(n)}$ that maps $\mathbf{x}$ to
$x_1 x_2 \dots x_i b x_{i+1} \dots x_n$
for some $0 \le i \le n$ and $b \in \{0, 1\}$
is called an insertion.
For a case $i=0$ (resp. $i=n$), the insertion maps $\mathbf{x}$
to $ b \mathbf{x}$ (resp. $\mathbf{x} b$).
Hence an insertion increases the length of a sequence.

An insertion sphere $\mathrm{iS}(\mathbf{x})$ for $\mathbf{x}$ is
a set of binary sequences defined as
$$
\mathrm{iS}( \mathbf{x} )
 := \{ I_{i, b}^{(n)} (\mathbf{x} ) \mid 0 \le i \le n, b \in \{0, 1\} \}.
$$
In other words, $\mathrm{iS}( \mathbf{x} )$ is the set of sequences
which are obtained by all (single) insertions to $\mathbf{x}$.

\begin{example}
Insertion spheres $\mathrm{iS}( 000 )$ and $\mathrm{iS}( 010 )$ are
\begin{align*}
\mathrm{iS}( 000 ) &= \{ 0000, 1000, 0100, 0010, 0001 \},\\
\mathrm{iS}( 010 ) &= \{ 0010, 0100, 1010, 0110, 0101 \}.
\end{align*}
\end{example}
While there are $2 (n+1)$ insertions,
the cardinality of $\mathrm{iS}( 000 )$ is not $8 (=2 \times 4)$ but $5$.
It follows from $I_{0, 0}^{(3)}(000) = I_{1, 0}^{(3)}(000) = I_{2, 0}^{(3)}(000)
 = I_{3, 0}^{(3)}(000) = I_{4, 0}^{(3)}(000) = 0000$.

By allowing multiplicity, an insertion sphere $\mathrm{iS}( \mathbf{x} )$ consists of
the following $2n + 2$ sequences:
\begin{align*}
 I_{n,   0}^{(n)} (\mathbf{x})= & x_1 x_2 \dots x_n 0 ,\\
 I_{n-1, 0}^{(n)} (\mathbf{x})= & x_1 x_2 \dots 0 x_n ,\\ 
             & \vdots \\
 I_{1,   0}^{(n)} (\mathbf{x})= & x_1 0 x_2 \dots x_n ,\\ 
 I_{0,   0}^{(n)} (\mathbf{x})= & 0 x_1 x_2 \dots x_n ,\\ 
 I_{0,   1}^{(n)} (\mathbf{x})= & 1 x_1 x_2 \dots x_n ,\\ 
 I_{1,   1}^{(n)} (\mathbf{x})= & x_1 1 x_2 \dots x_n ,\\ 
             & \vdots \\
 I_{n-1, 1}^{(n)} (\mathbf{x})= & x_1 x_2 \dots 1 x_n ,\\ 
 I_{n,   1}^{(n)} (\mathbf{x})= & x_1 x_2 \dots x_n 1.
\end{align*}

From here, let us define a set of operations to define insertions and an insertion sphere 
in terms of Weyl groups of type $B_n$.
Recall that $\lambda_B = (1/2, 1/2, \dots, 1/2) \in \mathbb{R}^n$ and
$\mathcal{M}_n$ is the set of $\lambda_B$-minuscule elements in $W(B_n)$.
The orbit $W( B_n ) \lambda_B$ is equal to a set $\{ 1/2, -1/2 \}^n$
 and hence the cardinality is $2^n$.
Let us introduce a map $\iota$ that maps a real vector $v$ to $v (1/2)$.
The map is regarded as an embedding map from $W( B_n ) \lambda_B $ to $W( B_{n+1} ) \lambda_B$:
$\lambda_1 \lambda_2 \dots \lambda_n \mapsto \lambda_1 \lambda_2 \dots \lambda_n 1/2$.

For $0 \le j \le 2n+1$, we denote the following element in $W(B_{n+1})$ by
$\mathcal{I}_j^{(n)}$:
\begin{align*}
\mathcal{I}_0^{(n)} &:= \mathrm{id},\\
\mathcal{I}_j^{(n)} &:= s_{n+2 -j} \dots s_{n} s_{n+1} & \text{for } 1 \le j \le n,\\
\mathcal{I}_j^{(n)} &:= s_{j - n} \dots s_{2} s_{1} \mathcal{I}_n^{(n)} & \text{for } n+1 \le j \le 2n+1.
\end{align*}
In other words, 
$\mathcal{I}_{j}^{(n)}$ is the consecutive right subword of 
$$s_{n} s_{n-1} \dots s_{1} s_{0} s_{1} \dots s_{n-1} s_n$$
of length $j$.

By combining $\iota$ and $\mathcal{I}^{(n)}_j$, we obtain
\begin{align*}
 \mathcal{I}_0^{(n)} \circ \iota (\lambda)= & \lambda_1 \lambda_2 \dots \lambda_n (1/2),\\
 \mathcal{I}_1^{(n)} \circ \iota (\lambda)= & \lambda_1 \lambda_2 \dots (1/2) \lambda_n ,\\ 
             & \vdots \\
 \mathcal{I}_{n-1}^{(n)} \circ \iota (\lambda)= & \lambda_1 (1/2) \lambda_2 \dots \lambda_n ,\\ 
 \mathcal{I}_n^{(n)} \circ \iota (\lambda)= & (1/2) \lambda_1 \lambda_2 \dots \lambda_n ,\\ 
 \mathcal{I}_{n+1}^{(n)} \circ \iota (\lambda)= & (-1/2) \lambda_1 \lambda_2 \dots \lambda_n ,\\ 
 \mathcal{I}_{n+2}^{(n)} \circ \iota (\lambda)= & \lambda_1 (-1/2) \lambda_2 \dots \lambda_n ,\\ 
             & \vdots \\
 \mathcal{I}_{2n}^{(n)} \circ \iota (\lambda)= & \lambda_1 \lambda_2 \dots (-1/2) \lambda_n ,\\ 
 \mathcal{I}_{2n+1}^{(n)} \circ \iota (\lambda)= & \lambda_1 \lambda_2 \dots \lambda_n (-1/2).
\end{align*}
Remember that we have seen a bijection $f_{01} \circ f_{1/2}^{-1}$ from $\{ 1/2, -1/2 \}^n$ to $\{0, 1\}^n$.
The observation above implies the following statement:
\begin{lemma}\label{lemma:com_bit_lambda}
For any $0 \le i \le n$ and $b \in \{0, 1\}$,
$$
I_{i, b}^{(n)} \circ (f_{01} \circ f_{1/2}^{-1})
 = (f_{01} \circ f_{1/2}^{-1}) \circ (\mathcal{I}_{i'}^{(n)} \circ \iota),
$$
where $i' = n + b + (-1)^{b+1} i$.
\end{lemma}

A natural action of $W(B_n)$ on $W(B_n) / W(A_{n-1})$
is defined as
$$
v [w] := [v w],
$$
for $v \in W(B_{n})$ and $[w] \in W(B_n) / W(A_{n-1})$.
The action inspires us to introduce an action of $W(B_n)$ on $\mathcal{M}_n$
as
$$ v w := \text{ the minimal coset representative of $[v w]$.}$$
Note that this action makes a map $f_{M} \circ f_{/}^{-1}$ homomorphic, i.e.,
$$ v (f_{M} \circ f_{/}^{-1}) ([w]) = (f_{M} \circ f_{/}^{-1}) ([v w]).$$

\begin{remark}
For positive integers $n$ and $m$ with $n \le m$,
$\Pi (B_{n})$ can be regarded as a subset of $\Pi (B_{m})$.
Hence $W(B_n)$ is regarded as a subgroup of $W( B_m )$.
Similarly, $\mathcal{M}_{n}$ is regarded as a subset of $\mathcal{M}_m$
since
a $\lambda_{B}$-minuscule element $w \in \mathcal{M}_{n}$ 
is also $\iota^{m-n} (\lambda_{B})$-minuscule.
\end{remark}

\begin{theorem}\label{theorem:main1}
For any $0 \le i \le n$ and any $b \in \{0,1\}$,
$$
\mathcal{I}_{i'}^{(n)} = (f_{01} \circ f_{M}^{-1})^{-1} \circ I_{i, b}^{(n)} \circ (f_{01} \circ f_{M}^{-1}),
$$
where $i' = n + b + (-1)^{b+1} i$.
\end{theorem}
\begin{proof}
The statement is equivalent to 
$$
I_{i, b}^{(n)} \circ (f_{01} \circ f_{M}^{-1})
= 
(f_{01} \circ f_{M}^{-1}) \circ \mathcal{I}_{i'}^{(n)}.
$$
For clarifying the argument below,
the length $n$ of $\lambda_{B}$ is described
as $\lambda_{B}^{(n)}$ here.

Let $w$ be an element of $\mathcal{M}_n$
and set $\lambda := w \lambda_{B}^{(n)}.$
\begin{align*}
& I_{i, b}^{(n)} \circ (f_{01} \circ f_{M}^{-1}) (w)
 & \\
&= I_{i, b}^{(n)} \circ (f_{01} \circ f_{1/2}^{-1}) \circ (f_{1/2} \circ f_{M}^{-1}) (w)
 & \\
&= I_{i, b}^{(n)} \circ (f_{01} \circ f_{1/2}^{-1}) (\lambda)
 & \text{ (Lemma \ref{lemma:bit_bijections}) } \\
&= (f_{01} \circ f_{1/2}^{-1}) \circ (\mathcal{I}_{i'}^{(n)} \circ \iota) (\lambda)
 & \text{ (Lemma \ref{lemma:com_bit_lambda}) }\\
&= (f_{01} \circ f_{1/2}^{-1}) \circ (\mathcal{I}_{i'}^{(n)} \circ \iota) (w \lambda_{B}^{(n)})
 & \text{(Def. of $\lambda$)}\\
&= (f_{01} \circ f_{1/2}^{-1}) \circ \mathcal{I}_{i'}^{(n)} \circ w (\iota \lambda_{B}^{(n)})
 & \text{($w \in W(B_n)$)}\\
&= (f_{01} \circ f_{1/2}^{-1}) \circ \mathcal{I}_{i'}^{(n)} \circ w (\lambda_{B}^{(n+1)})
 & \text{(Def. of $\iota$)}\\
&= (f_{01} \circ f_{1/2}^{-1}) \circ (\mathcal{I}_{i'}^{(n)} w) (\lambda_{B}^{(n+1)})
 & \\
&= (f_{01} \circ f_{1/2}^{-1}) \circ (f_{1/2} \circ f_{M}^{-1}) (\mathcal{I}_{i'}^{(n)} w)
 & \text{(Lemma \ref{lemma:bit_bijections})}\\
&= (f_{01} \circ f_{M}^{-1}) (\mathcal{I}_{i'}^{(n)} w)
 & \\
&= (f_{01} \circ f_{M}^{-1}) \circ \mathcal{I}_{i'}^{(n)} (w).
 &
\end{align*}
\end{proof}

Theorem \ref{theorem:main1} encourages us to call $\mathcal{I}_i^{(n)}$
an insertion for $\mathcal{M}_n$ (to $\mathcal{M}_{n+1}$).
Similarly we define an insertion sphere for $w \in \mathcal{M}_n$ as
$$
\mathrm{iS}(w)
 :=
   \{ \mathcal{I}_j^{(n)} w \in \mathcal{M}_{n+1} \mid 0 \le j \le 2n+1 \}.
$$

\section{Deletions and Perfect Codes}\label{sec:delAndPerfect}
Opposite operations to insertions are deletions.
For a binary sequence $\mathbf{x} := x_1 x_2 \dots x_n$,
an operation $D_{i}$ that maps $\mathbf{x}$ to
$x_1 x_2 \dots x_{i-1} x_{i+1} \dots x_n$
for $1 \le i \le n$
is called a deletion.
Hence a deletion decreases the length of a sequence.

A deletion sphere $\mathrm{dS}(\mathbf{x})$ for $\mathbf{x}$ is
a set of binary sequences defined as
$$
\mathrm{dS}( \mathbf{x} )
 := \{ D_{i} (\mathbf{x} ) \mid 1 \le i \le n \}.
$$
In other words, $\mathrm{dS}( \mathbf{x} )$ is the set of sequences
which is obtained by all (single) deletions to $\mathbf{x}$.

\begin{example}
Deletion spheres $\mathrm{dS}( 0000 )$ and $\mathrm{dS}( 0101 )$ are
\begin{align*}
\mathrm{dS}( 0000 ) &= \{ 000 \},\\
\mathrm{dS}( 0101 ) &= \{ 101, 001, 011, 010 \}.
\end{align*}
\end{example}

The next lemma shows that
a deletion sphere can be defined by using an insertion sphere.
\begin{lemma}\label{lemma:dSandiS}
For any positive integer $n$ and any binary sequence $\mathbf{x} \in \{0,1\}^n$,
$$
\mathrm{dS}(\mathbf{x})
 =
\{ \mathbf{y} \in \{0,1 \}^{n-1} \mid \mathbf{x} \in \mathrm{iS}(\mathbf{y}) \}.
$$
\end{lemma}
\begin{proof}
Let us set $\mathbf{x} = x_1 x_2 \dots x_n$.

($\mathrm{dS}(\mathbf{x}) \subset \text{R.H.S.}$) : 
For any $\mathbf{y} \in \mathrm{dS}(\mathbf{x})$,
there exists $i$ such that $\mathbf{y} = D_i ( \mathbf{x} )$.
In other words,
$\mathbf{y} = x_1 x_2 \dots x_{i-1} x_{i+1} \dots x_n$.
Then $\mathbf{x} = I_{i-1, x_{i} } ( \mathbf{y} )$.
In other words,
$\mathbf{x} \in \mathrm{iS}( \mathbf{y} )$.
Hence $\mathbf{y} \in \text{R.H.S.}$

($\mathrm{dS}(\mathbf{x}) \supset \text{R.H.S.}$) :
The assumption $\mathbf{x} \in \mathrm{iS}( \mathbf{y} )$
implies that 
there exists $i$ and $b$ such that $I_{i, b}( \mathbf{y} ) = \mathbf{x}$.
In other words,
$y_1 y_2 \dots y_i b y_{i+1} \dots y_{n-1}
=
x_1 x_2 \dots x_i x_{i+1} x_{i+2} \dots x_{n}$.
Then $D_{i+1}( \mathbf{x} ) = \mathbf{y}$.
Hence $\mathbf{y} \in \mathrm{dS}( \mathbf{x} )$.
\end{proof}

\begin{definition}[Perfect Codes for Deletions]
A set $C \subset \{0, 1\}^n$ is called a perfect code
(for deletions) if 
\begin{itemize}
\item for any different $\mathbf{c} , \mathbf{c}' \in C$, $\mathrm{dS}(\mathbf{c}) \cap \mathrm{dS}(\mathbf{c}') = \emptyset$,
\item $\bigcup_{\mathbf{c} \in C} \mathrm{dS}( \mathbf{c} ) = \{0, 1\}^{n-1}$.
\end{itemize}
In other words, $\{ \mathrm{dS}( \mathbf{c} ) \mid \mathbf{c} \in C \}$ is a partition
of $\{0, 1\}^{n-1}$.
\end{definition}

The following is proven by Levenshtein.
\begin{theorem}[Theorem 2.1 in \cite{levenshtein1992perfect}]\label{thm:perfectBinaryCodes}
For any positive integer $n$ and any integer $a$,
$\mathbb{L}_{n, a}$ is a perfect code for deletions.
\end{theorem}

A sketch of a proof different from \cite{levenshtein1992perfect}
is given here.
\begin{proof}
Let $\psi$ be the inverse of a function from $\{ 0,1,\dots,n-1 \} \times \{0, 1\}$
to $\{0, 1, \dots, 2n-1 \}$: $(i, b) \mapsto n + b + (-1)^{b+1} i$.
Then for any $\mathbf{x} \in \{ 0,1 \}^{n-1}$, it is easy to see
$$
\rho( I^{(n-1)}_{\psi( j + 1)} ( \mathbf{x} ) )=
\begin{cases}
\rho( I^{(n-1)}_{\psi( j )} ( \mathbf{x} ) ) & \text{ if } I^{(n-1)}_{\psi( j + 1)} ( \mathbf{x} ) = I^{(n-1)}_{\psi( j )} ( \mathbf{x} ) \\
\rho( I^{(n-1)}_{\psi( j )} ( \mathbf{x} ) ) + 1 & \text{ otherwise }
\end{cases}
$$
and
$$
\rho( I^{(n-1)}_{\psi( 2n - 1)} ( \mathbf{x} ) ) - \rho( I^{(n-1)}_{\psi( 0 )} ( \mathbf{x} ) )
=
\rho(  \mathbf{x} 1 ) - \rho(  \mathbf{x} 0 ) = n.
$$
Hence $\{  \rho( I^{(n-1)}_{\psi( j )} ( \mathbf{x} ) ) \mid 0 \le j \le 2n - 1 \}$
is a set of consecutive integers $\rho(  \mathbf{x}), \rho(  \mathbf{x}) + 1 , \dots ,\rho(  \mathbf{x}) + n$.
In other words, for any $a$, there uniquely exists $\mathbf{y} \in \mathrm{iS}( \mathbf{x} )$
such that $\mathbf{y} \in \mathbb{L}_{n, a}$.
\end{proof}

Lemma \ref{lemma:dSandiS} enables us to rewrite the definition of
perfect codes by using $\mathrm{iS}()$ instead of $\mathrm{dS}()$:
$$ \mathrm{dS}( w ) := \{ y \in M_{n-1} \mid w \in \mathrm{iS}( y ) \}.$$
Remember that insertion spheres $\mathrm{iS}()$ are defined in terms of 
Weyl groups.
Hence perfect codes are defined in terms of Weyl groups.
\begin{definition}[Perfect Codes]
A set $C \subset \mathcal{M}_n$ is called a perfect code
if 
\begin{itemize}
\item for any different $w , w' \in C$, $\mathrm{dS}(w) \cap \mathrm{dS}(w') = \emptyset$,
\item $\bigcup_{w \in C} \mathrm{dS}( w ) = \mathcal{M}_{n-1}$.
\end{itemize}
In other words, $\{ \mathrm{dS}( w ) \mid w \in C \}$ is a partition
of $\mathcal{M}_{n-1}$.
\end{definition}

\begin{corollary}
For any positive $n$ and any integer $a$,
$\mathcal{M}_{n, a}$ is a perfect code.
\end{corollary}

\section{Generalizations for Insertions and Deletions}\label{sec:genInsDel}
Based on the previous sections,
a generalization for insertions is proposed as follows:
\begin{definition}[Generalized Insertions]\label{def:generalizedInsertions}
Let $X$ and $Y$ be finite sets,
$f$ a function on $Y$,
and 
$I_j$ maps from $X$ to $Y$ for $j = 0, 1, \dots, r$ for some $r$.

$I_j$ ($0 \le j \le r$) are called generalized insertions if
\begin{itemize}
\item[(I1)] for any $x \in X$ and for any $0 \le j < r$,
$$
f \circ I_{j+1} (x) = f \circ I_j (x) + 1 \text{ if } I_{j+1} (x) \neq I_j (x),
$$
\item[(I2)] there exists an integer $S$ such that for any $x \in X$,
$$
S = f \circ I_r (x) - f \circ I_0 (x).
$$
\end{itemize}

In this section, $\mathbf{I}$ denotes the set of generalized insertions,
i.e.
$$ \mathbf{I} = \{ I_j \mid 0 \le j \le r \}.$$
\end{definition}

\begin{example}\label{example:originalInsertion}
The original insertions on bit sequences
satisfy the conditions in Definition \ref{def:generalizedInsertions}
by setting
$ X := \{ 0, 1 \}^n,     $
$ Y := \{ 0, 1 \}^{n+1},$
$ f     := \rho,            $
$ I_j   := I_{ \psi(j) },   $
and
$ S     := n+1.$

Conditions (I1) and (I2) are mentioned in the sketch of a proof for perfectness
(see the Proof of Theorem \ref{thm:perfectBinaryCodes}).
\end{example}

\begin{example}
By Theorem \ref{theorem:main1}, Example \ref{example:originalInsertion}
and setting
$ X := \mathcal{M}_n,     $
$ Y := \mathcal{M}_{n+1},$
$ f     := l,            $
$ I_j   := \mathcal{I}_j,   $
and
$ S     := n+1,$
$\mathcal{I}_j$ are generalized insertions.
\end{example}

A generalization of deletions is also proposed here.
\begin{definition}\label{def:generalizedDeletions}
Under the same notation in Definition \ref{def:generalizedInsertions},
let $\mathbf{D}$ be a set of partial maps from $ Y $ to $ X $.
If the set $\mathbf{D}$  satisfies the following two conditions, elements
of $\mathbf{D}$ are called generalized deletions.
\begin{itemize}
\item[(D1)]:
For any $D \in \mathbf{D}$ and any $y \in Y$, there exists $I \in \mathbf{I}$
such that
$$I \circ D (y) = y.$$

\item[(D2)]:
For any $I \in \mathbf{I}$ and any $x \in X$,
there exists $D \in \mathbf{D}$ such that
$$ D \circ I (x) = x.$$
\end{itemize}
\end{definition}

\begin{lemma}\label{lemma:relInsAndDel}
Let $x \in X$ and $y \in Y$.
The following are equivalent.
\begin{itemize}
\item[(1)] $x = D (y)$ for some $D \in \mathbf{D}$.
\item[(2)] $y = I (x)$ for some $I \in \mathbf{I}$.
\end{itemize}
\end{lemma}
\begin{proof}
( (1) $\Rightarrow$ (2) ) :
Assume $x = D (y)$.
By (D1),
there exists an insertion $I$ such that
$I \circ D (y) = y$.
Therefore
$$ y = I \circ D (y) = I (x).$$

( (2) $\Rightarrow$ (1) ) :
Assume $y = I(x)$.
By (D2),
there exists a deletion $D$ such that
$D \circ I (x) = x$.
Therefore
$$ x = D \circ I (x) = D (y).$$
\end{proof}

\begin{definition}[Insertion Sphere and Deletion Sphere]
Let us define an insertion sphere $\mathrm{iS}()$ and a deletion sphere $\mathrm{dS}()$  as follows:
$$
\mathrm{iS}(x) := \{ I (x) \in Y \mid I \in \mathbf{I} \},
$$
$$
\mathrm{dS}(y) := \{ w \in X \mid y \in \mathrm{iS}(w) \},
$$
where $x \in X$ and $y \in Y$.
\end{definition}

Remark that Lemma \ref{lemma:relInsAndDel} implies
$$
\mathrm{dS}( y ) = \{ D( y ) \mid D \in \mathbf{D}, \text{ if $D(y)$ is defined.} \}.
$$

Similar to the argument in the proof for Theorem \ref{thm:perfectBinaryCodes},
the following holds.
\begin{lemma}\label{lemma:imf}
1) For any $x \in X$, 
$\{ f(y) \mid y \in \mathrm{iS}( x ) \}$ is a set of
consecutive integers $f \circ I_0 (x), f  \circ I_0 (x) + 1, \dots, f  \circ I_0 (x)  + S$.

2) For any $x \in X$, $\# \mathrm{iS}( x ) = S + 1$.

3) For any $x \in X$, a restriction $f_{ | \mathrm{iS}(x) }$ of $f$ to $\mathrm{iS}(x)$ is injective.
\end{lemma}

For any integer $a$, let us set
$$
C_a := \{ y \in Y \mid f(y) \equiv a \pmod{ S + 1 } \}.
$$

\begin{lemma}\label{generalizedPerfectCodes}
For any integer $a$,
the following set $C_a$ is a perfect code for generalized deletions.
\end{lemma}
\begin{proof}
First we show
$\mathrm{dS}(c ) \cap \mathrm{dS}( c' ) \neq \emptyset$ implies
$c = c'$ where $c, c' \in C_a$.
Assume there exists $x \in \mathrm{dS}(c ) \cap \mathrm{dS}( c' )$.
By the definition of deletion sphere, $c, c' \mathrm{iS}( x )$.
In other words, there exists insertions $I, I'$ such that
$c = I(x)$ and $c' = I' (x)$.

As is mentioned in Lemma \ref{lemma:imf},
$\{ f(w) \mid w \in \mathrm{iS}(x) \}$ is a set
of consecutive $S+1$ integers
and that has a unique element that is equivalent to $a$ module $S+1$.
Therefore $f(c) = f(c')$.
By the injectivity of $f_{| \mathrm{iS}(x)}$,
$c = c'$ holds.

Second we show $X = \cup_{c \in C_a} \mathrm{dS}( c )$.
As is shown, for any $x \in X$ there exists $c \in C_a$ such that
$c \in \mathrm{iS}(x)$.
By the definition of a deletion sphere, it is equivalent to
$x \in \mathrm{dS}(c)$.
Hence for any $x \in X$, $x \in \cup_{c \in C_a} \mathrm{dS}(c)$.
Namely $X \subset \cup_{c \in C_a} \mathrm{dS}(c)$.
Therefore, $X = \cup_{c \in C_a} \mathrm{dS}(c)$.
\end{proof}
\section{Constant Hamming Weight Sequences and Minuscule Elements of Type $A$}\label{consTypeA}
For positive integers $v$ and $h$,
$\mathbb{R}^{v, h}$ denotes a real vector space
of dimension $v+h$.
Here the index of an orthogonal basis $\langle \epsilon_i \rangle$ of $\mathbb{R}^{v, h}$
starts with $-v$ and ends with $h-1$, i.e.,
$\epsilon_{-v}, \epsilon_{-(v-1)}, \dots, \epsilon_0 , \dots, \epsilon_{h-2}, \epsilon_{h-1}$.
For $-(v-1) \le i \le h-1$, set simple roots as $\alpha_i := \epsilon_{i-1} - \epsilon_{i}$.
Hence there are $v+h-1$ roots $\alpha_{-(v-1) }, \dots, \alpha_{0}, \dots, \alpha_{h -1}$.
For $v, h, u$ and $k$, a space $\mathbb{R}^{v, h}$ is regarded as a subspace
of $\mathbb{R}^{u, k}$ if both $v \le u$ and $h \le k$ hold.
A subset $\Pi (A_{v, h})$ of $\mathbb{R}^{v, h}$ forms a simple system of type $A_{v+h-1}$.
Let $W( A_{v, h} )$ denote a Weyl group generated by $\{ s_{\alpha} \}_{\alpha \in \Pi (A_{v, h} )}$
and $W_0 ( A_{v, h} )$ a subgroup generated by $\{ s_{\alpha} \}_{\alpha \in \Pi (A_{v, h} ) \setminus \{ \alpha_0 \}} $.
Hence $W( A_{v, h})$ is of type $A_{v+h-1}$
and $W_0 (A_{v, h})$ is isomorphic to a direct product of two Weyl groups of type $A_{v-1}$ and $A_{h-1}$.
Furthermore, $W_{0} (A_{v, h})$ is a maximal parabolic subgroup of $W( A_{v, h} )$.

\begin{remark}\label{constructionOfMinusculeElementsOfTypeA}
Set $\lambda_{v, h} := \sum_{-v \le j \le -1} (1/2 \epsilon_j )  + \sum_{0 \le i \le h-1} (-1/2 \epsilon_i )$,
for positive integers $v$ and $h$.
In other words, 
$$\lambda_{v, h} 
= (\overbrace{1/2, 1/2, \dots, 1/2}^{v}, \overbrace{-1/2, -1/2, \dots, -1/2}^{h}).
$$

For an $v$-elements subset $J := \{ j_{-v}, j_{-(v-1)}, \dots, j_{-1} \}$ of 
an $v+h$-elements subset $\{-v, -(v-1), \dots, h-2, h-1 \}$,
where $j_{-v} < j_{-(v-1)} < \dots < j_{-1}$,
let us define an element $w_J \in W(A_{v,h})$ as follows:
$$
w_J := w_{-v, j_{-v}} w_{-(v-1), j_{-(v-1)}} \dots w_{-1, j_{-1}},
$$
where
$$ w_{k, j_k} := s_{j_{k} } s_{j_{k} - 1 } \dots s_{k + 2 } s_{k + 1} $$
for $-v \le k \le -1$.
Note that if $ j_k \le k $, $w_{k, j_k}$ is defined as the identity element.
Indeed, a reduced expression of $w_J$ is obtained if we replace
all of $w_{k, j_k}$ with $s_{j_{k} } s_{j_{k} - 1 } \dots s_{k + 2 } s_{k + 1} $ from 
$w_J = w_{-v, j_{-v}} w_{-(v-1), j_{-(v-1)}} \dots w_{-1, j_{-1}}$. (See \cite{stembridge2001minuscule}.)
Then $w_J$ is a $\lambda_{v, h}$-minuscule element.
On the entries of a vector $\lambda_{-v} \lambda_{-v+1} \dots \lambda_{h-2} \lambda_{h-1} := w_J \lambda_{v,h}$,
the entry $\lambda_j$ is
$1/2$ if $j \in J$ and $-1/2$ otherwise.
Hence $\binom{v+h}{v}$ $\lambda_{v,h}$-minuscule elements are obtained by the construction above.
\end{remark}

Let $\mathcal{Y}_{v,h}$ denote the set of $\lambda_{v,h}$-minuscule elements constructed from Remark \ref{constructionOfMinusculeElementsOfTypeA}.
By a similar argument for $\mathcal{M}_n$ to Section \ref{sec:BitSeqMinElemTypeB},
it is easy to show that $\mathcal{Y}_{v, h}$ is the set of minimal coset representatives
of a right coset $W (A_{v, h}) / W_0 (A_{v, h})$.
Furthermore bijections are defined among the following sets of cardinality $\binom{v+h}{v}$.
\begin{itemize}
\item $\mathbb{Y}_{v, h} := \{ \mathbf{x} \in \{0, 1\}^{v + h} \mid \mathrm{wt}( \mathbf{x} ) = h \}$.
Note that the minimum (resp. maximum) index is $-v$ (resp. $h-1$).
\item $W( A_{v, h} ) \lambda_{v, h}
 = \{ \lambda \in \{1/2, -1/2\}^{v + h}
 \mid \text{there are $v$-entries of $1/2$ in $\lambda$}
   \}$,
\item $\mathcal{Y}_{v, h}$,
\item $W(A_{v, h}) / W_0 (A_{v, h})$.
\end{itemize}

Here four bijections whose domain is the power set $\mathcal{J}_n$ of a set $\{-v,-(v-1), \dots, h-1 \}$ are
introduced as follows:
\begin{itemize}
\item $g_{01} : \mathcal{J}_n \to \mathbb{Y}_{v + h}$,
$$ g_{01} ( J ) := x_{-v} x_{-(v-1)} \dots x_{h-1},$$
where $x_i := 0$ if $i \in J$ and $x_i := 1$ otherwise.
\item $g_{1/2} : \mathcal{J}_n \to \mathcal{Y}_{v, h} \lambda_{v, h}$,
$$ g_{1/2} ( J ) := \lambda_{-v} \lambda_{-(v-1)} \dots \lambda_{h-1},$$
where $\lambda_i := 1/2$ if $i \in J$ and $\lambda_i := -1/2$ otherwise.
\item $g_{M} : \mathcal{J}_n \to \mathcal{Y}_{v, h}$,
$$ g_{M} ( J ) := w_J,$$
where $g_J$ is defined in Remark \ref{constructionOfMinusculeElementsOfTypeA}.
\item $g_{/} : \mathcal{J}_n \to W(A_{v, h}) / W_0 (A_{v, h})$,
$$ g_{/} ( J ) := w_J W_{0}(A_{v, h}).$$
\end{itemize}

\begin{example}\label{example:JandMinuscule}
Set $v := 4, h := 5$ and $J := \{-4, -1, 0 ,4 \}$.
Then
\begin{align*}
g_{01} (J)  & = 011001110,\\
g_{M}  (J)  & = (\mathrm{id}) (s_{-1} s_{-2}) (s_{0} s_{-1}) (s_4 s_3 s_2 s_1 s_0 ).
\end{align*}
\end{example}

From here, for simplicity, $g_{01}(J)$ is denoted by $\mathbf{x}_J$.
\begin{remark}
Minuscule elements are related to minuscule heaps and d-complete posets
\cite{proctor1999dynkin,proctor1999minuscule,stembridge2001minuscule}.
Minuscule heaps are labeled posets whose label set is the simple system.
If the simple system is of type $A$, the minuscule heaps are
related to Young diagram.

Without the definition of minuscule heap,
the related heap for
$g_M (J)$ of Example \ref{example:JandMinuscule} is shown in
Figure \ref{figure:expl_typeA}.
Readers will find a relation between 
the minuscule heap and 
the reduced expression $(s_{-1} s_{-2}) (s_{0} s_{-1}) (s_4 s_3 s_2 s_1 s_0)$
from the figure.

\begin{figure}[htbp]
\begin{center}
\includegraphics[width=6cm, bb=0 0 198 141]{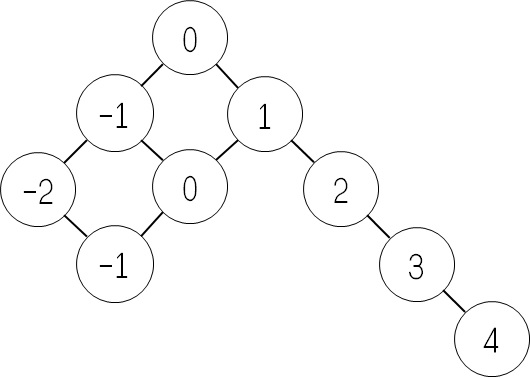} %
\caption{Minuscule Heap for $g_M (\{-4, -1, 0 ,4 \} )$}
\label{figure:expl_typeA}
\end{center}
\end{figure}

By rotating a minuscule heap and replacing circles with squares,
a Young tableau is obtained.
A Young tableau related to the minuscule heap for $g_M (\{-4, -1, 0 ,4 \} )$
is shown in Figure \ref{figure:expl_YoungAndPath}.

\begin{figure}[htbp]
\begin{center}
\includegraphics[width=10cm, bb=0 0 375 176]{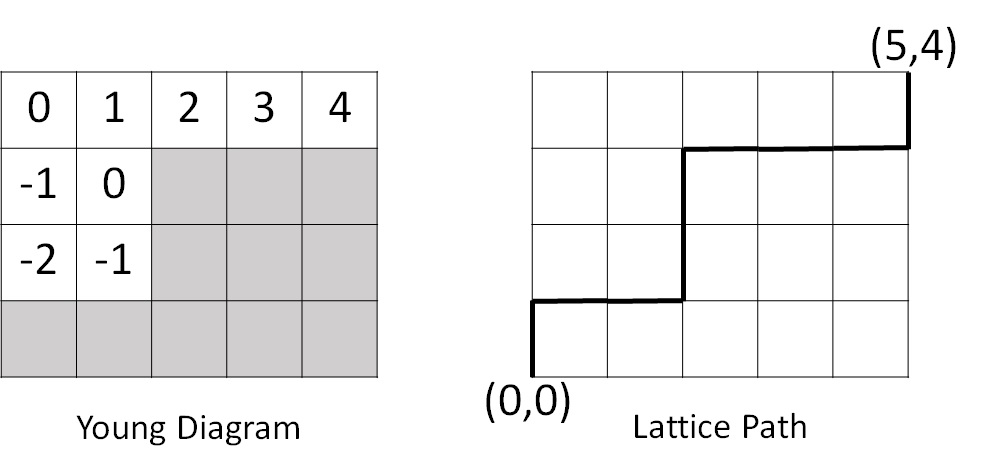} 
\caption{Related Young Diagram for $g_M (\{-4, -1, 0 ,4 \} )$ and Lattice Path for $x_{\{-4, -1, 0 ,4 \}} = 011001110$}
\label{figure:expl_YoungAndPath}
\end{center}
\end{figure}

Furthermore, by embedding a Young tableau into 
an area of a lattice rectangular from $(0,0)$ to $(h, v)$
and tracing its rim,
a lattice path is obtained.
Remark that the lattice path for $g_M (J)$ is related to $x_J$
for $J$.
The relation is bitwise replacement ``$0$'' with a vertical line ``$|$'' and ``$1$'' with
a horizontal line ``$-$''.
\end{remark}

\begin{lemma}
For any $J$,
$$l( w_J ) = \mathrm{inv} ( \mathbf{x}_J ),$$
where 
$\mathrm{inv} ( \mathbf{x}_J )
 = \# \{ (i , j) \mid -v \le i < j \le h-1, x_i > x_j \} $.
In other words, $\mathrm{inv}(\mathbf{x}_J)$ is equal to
the number of cells of the related Young diagram.
\end{lemma}
\begin{proof}
By the definition of $w_J$,
\begin{align*}
l(w_J)
 &= \sum_{-v \le k \le -1} \# \{ i \mid j_k \ge i > k \}\\
 &= \sum_{-v \le k \le -1} (j_k - k).
\end{align*}
On the other hand,
\begin{align*}
\mathrm{inv}(\mathbf{x}_J)
 &= \# \{ (i , j) \mid -v \le i < j \le h-1, x_i > x_j \} \\
 &= \# \{ (i , j) \mid -v \le i < j \le h-1, x_i = 1, x_j = 0 \} \\
 &= \# \{ (i , j) \mid -v \le i < j, x_i = 1, j \in J  \} \\
 &= \sum_{-v \le k \le -1} \# \{ i \mid -v \le i < j_k, x_i = 1 \}.
\end{align*}
By the definition of $\mathbf{x}_J$,
$x_{j_k}$ is the left $(-k)$th zero.
It implies that there are $v - (- k)$ zeros in $x_{-v} x_{-(v-1)} \dots x_{j_k -1}$.
Hence $ \# \{ i \mid -v \le i < j_k, x_i = 1 \} = (j_k - (-v)) - (v - (-k) ) = j_k - k$.
\end{proof}

\begin{definition}[$\mathbb{Y}_{v, h, a}$]
For any positive integers $v, h$ and any integer $a$,
let us set
$$\mathbb{Y}_{v, h, a} := \{ \mathbf{x} \in \mathbb{Y}_{v, h} \mid \mathrm{inv}(\mathbf{x}) \equiv a \pmod{v + h} \}.$$
\end{definition}
In the next section, $\mathbb{Y}_{v,h,a}$ is shown to be perfect for generalized insertions.

\section{Generalized Insertions and Deletions of Type A: Lattice Paths}\label{sec:perfectTypeA}

Here we propose a generalization of insertions
for $\mathbb{Y}_{v,h}$.
Remark that in Section \ref{sec:inserstionForM}, generalized insertions $\mathcal{I}_{j}^{(n)}$
are defined as a consecutive right subsequence of the following word
$$
s_{n} s_{n-1} \dots s_{1} s_{0} s_{1} \dots s_{n-1} s_n.
$$
In fact, the word is the unique reduced expression of 
a reflection of the highest coroot for $\Phi( B_{n+1} )$.
From this point of view, our generalized insertions for $\mathbb{Y}_{v, h}$
are introduced by using the highest coroot $\epsilon_{-(v+1)} - \epsilon_{h}$
for $\Phi( A_{v+1, h+1} )$.
As is different from a case of type $B$, a reduced expression
of the reflection is not unique.
However here, we focus on the following choice:
$$
s_{h} s_{h-1} \dots s_{-(v-1)} s_{-v} s_{-(v-1)} \dots s_{h-1} s_h
$$
By the choice of this reduced expression,
the related operations $\mathcal{K}^{(v,h)}_j$ ($0 \le j \le 2(v+h) + 1$)
are defined as
\begin{eqnarray*}
\mathcal{K}^{(v,h)}_0 &:=& \mathrm{id},\\
\mathcal{K}^{(v,h)}_{j} &:=& s_{h + 1 - j} \mathcal{K}^{(v,h)}_{j-1} \;\; (1 \le j \le v+h+1),\\
\mathcal{K}^{(v,h)}_{j} &:=& s_{j - 2v - h - 1} \mathcal{K}^{(v,h)}_{j-1} \;\; (v+h+2 \le j \le 2(v+h)+1).
\end{eqnarray*}
Additionally let $\kappa^{(v,h)}$ be a map from $\mathbb{Y}_{v,h}$ to $\mathbb{Y}_{v+1, h+1}$
that maps $\mathbf{y}$ to $0 \mathbf{y} 1$
and set $K^{(v,h)}_j := \mathcal{K}^{(v,h)}_j \circ \kappa^{(v,h)}$.

\begin{theorem}\label{thm:I1I2forK}
The operations $K^{(v,h)}_i$ satisfy the conditions (I1) and (I2)
for
$
X := \mathbb{Y}_{v,h},
Y := \mathbb{Y}_{v+1, h+1},
r := 2(v+h)+1,
f := \mathrm{inv},
$ and $S := v+h+2$
with notation in Definition \ref{def:generalizedInsertions} (Generalized Insertions).
In other words,
$\mathbf{K}^{(v,h)} := \{ K^{(v,h)}_j \mid 0 \le j \le 2(v+h) + 1 \}$
is a set of generalized insertions.
\end{theorem}
\begin{proof}
First we observe how $K^{(v,h)}_i$ acts on a sequence
$\mathbf{y} := y_{-v} y_{-(v-1)} \dots y_{h-2} y_{h-1}$:
\begin{align*}
K^{(v,h)}_0 (\mathbf{y} ) &= 0 y_{-v} y_{-(v-1)} \dots y_{h-2} y_{h-1} 1,\\
K^{(v,h)}_1 (\mathbf{y} ) &= 0 y_{-v} y_{-(v-1)} \dots y_{h-2} 1 y_{h-1},\\
K^{(v,h)}_2 (\mathbf{y} ) &= 0 y_{-v} y_{-(v-1)} \dots 1 y_{h-2} y_{h-1},\\
& \vdots \\
K^{(v,h)}_{v+h-2}   (\mathbf{y} ) &= 0 y_{-v} 1 y_{-(v-1)} \dots y_{h-2} y_{h-1},\\
K^{(v,h)}_{v+h-1} (\mathbf{y} ) &= 0 1 y_{-v} y_{-(v-1)} \dots y_{h-2} y_{h-1},\\
K^{(v,h)}_{v+h} (\mathbf{y} ) &= 1 0 y_{-v} y_{-(v-1)} \dots y_{h-2} y_{h-1},\\
K^{(v,h)}_{v+h+1} (\mathbf{y} ) &= 1 y_{-v} 0 y_{-(v-1)} \dots y_{h-2} y_{h-1},\\
K^{(v,h)}_{v+h+2} (\mathbf{y} ) &= 1 y_{-v} y_{-(v-1)} 0 \dots y_{h-2} y_{h-1},\\
& \vdots \\
K^{(v,h)}_{2(v+h)-1} (\mathbf{y} ) &= 1 y_{-v} y_{-(v-1)} \dots 0 y_{h-2} y_{h-1},\\
K^{(v,h)}_{2(v+h)} (\mathbf{y} ) &= 1 y_{-v} y_{-(v-1)} \dots y_{h-2} 0 y_{h-1},\\
K^{(v,h)}_{2(v+h)+1} (\mathbf{y} ) &= 1 y_{-v} y_{-(v-1)} \dots y_{h-2} y_{h-1} 0.
\end{align*}

For $0 \le j \le v+h-1$,
$K^{(v,h)}_{j}   (\mathbf{y} ) \neq K^{(v,h)}_{j+1}   (\mathbf{y} )$ is equivalent to
$y_{h+1-j} 1 \neq 1 y_{h+1-j}$, i.e., $y_{h+1-j} = 0$.
In this case, it is easy to check
$\mathrm{inv} \circ K^{(v,h)}_{j+1}   (\mathbf{y} ) = \mathrm{inv} \circ K^{(v,h)}_{j}   (\mathbf{y} )  + 1$.

For $j = v+h$,
as is observed above, $K^{(v,h)}_{v+h+1} (\mathbf{y} ) \neq K^{(v,h)}_{v+h+2} (\mathbf{y} )$ holds.
In this case, it is also easy to check
$\mathrm{inv} \circ K^{(v,h)}_{j+1}   (\mathbf{y} ) = \mathrm{inv} \circ K^{(v,h)}_{j}   (\mathbf{y} )  + 1$.

For $v+h+1 \le j <  2(v+h)+1$,
$K^{(v,h)}_{j}   (\mathbf{y} ) \neq K^{(v,h)}_{j+1}   (\mathbf{y} )$ is equivalent to
$0 y_{j - 2v - h -2 } \neq y_{j - 2v - h -2} 0$, i.e., $y_{j - 2v - h -2 } = 1$.
In this case, it is also easy to check
$\mathrm{inv} \circ K^{(v,h)}_{j+1}   (\mathbf{y} ) = \mathrm{inv} \circ K^{(v,h)}_{j}   (\mathbf{y} )  + 1$.

Hence (I1) holds.

For showing (I2),
let us calculate $\mathrm{inv} \circ K^{(v,h)}_{0}   (\mathbf{y} )$ and $\mathrm{inv} \circ K^{(v,h)}_{2(v+h)+1}   (\mathbf{y} )$.
Set $\mathbf{x} := K^{(v,h)}_{0}   (\mathbf{y} )$
and $\mathbf{z} := K^{(v,h)}_{2(v+h)+1}   (\mathbf{y} )$.
By the definition of $\mathrm{inv}$,
\begin{align*}
 & \mathrm{inv} ( \mathbf{x} ) & \\
 = &\# \{ (i , j) \mid -(v+1) \le i < j \le h, x_i > x_j \}  & \\
 = &\# \{ (i , j) \mid -v \le i < j \le h-1 , x_i > x_j \}  & (x_{-(v+1)} = 0, x_{h} = 1)\\
 = &\# \{ (i , j) \mid -v \le i < j \le h-1 , y_i > y_j \} & (x_{j} = y_j, \text{ for } -v \le j \le h-1)\\
 = &\mathrm{inv} ( \mathbf{y} ), &
\end{align*}
and
\begin{align*}
 & \mathrm{inv} ( \mathbf{z} ) & \\
 = &\# \{ (i , j) \mid -(v+1) \le i < j \le h, z_i > z_j \}  & \\
 = &    \# \{ (i , j) \mid -v \le i < j \le h-1 , z_i > z_j \}  & \\
   & +  \# \{ (-(v+1) , j) \mid -v \le j \le h-1 , 1 > z_j \}   & \\
   & +  \# \{ (i , h) \mid -v \le i \le h-1 , z_i > 0 \}     & \\
   & +  \# \{ (-(v+1), h) \}                                 & (z_{-(v+1)} = 1, z_{h} = 0)\\
 = &    \# \{ (i , j) \mid -v \le i < j \le h-1 , z_i > z_j \}  & \\
   & +  \# \{ j \mid -v \le j \le h-1 , z_j = 0 \} & \\
   & +  \# \{ i \mid -v \le i \le h-1 , z_i = 1 \} & \\
   & +  1                                 &  (z_i > z_j \iff z_i=1, z_j=0 )\\
 = & \# \{ (i , j) \mid -v \le i < j \le h-1 , y_i > y_j \} \\
   &  + (v + h + 1) + 1 & (z_{j} = y_j, \text{ for } -v \le j \le h-1)\\
 = &\mathrm{inv} ( \mathbf{y} ) + v + h + 2.&
\end{align*}
Hence (I2) holds.
\end{proof}

By the observation for actions on a bit sequence by $\mathbf{K}^{(v,h)}$ in the proof above,
$\mathbf{K}^{(v,h)}$ is the set of standard two insertions of bits $x_i, x_j \; (i < j)$
with $i=1$ and $x_i \neq x_j$.
Then a set $\mathbf{E}^{(v,h)}$ is defined as the set of standard two deletions
of the first bit $y_1$ and another bit $y_j$ with $y_1 \neq y_i$.
Conditions (D1) and (D2) in Definition \ref{def:generalizedDeletions}
clearly holds on $\mathbf{E}^{(v,h)}$ with $\mathbf{K}^{(v,h)}$.
By Lemma \ref{generalizedPerfectCodes},
$\mathbb{Y}_{v+1,h+1,a}$ is perfect by $\mathbf{E}^{(v,h)}$.
Regarding $\mathbb{Y}_{v+1,h+1,a}$
as a set of lattice paths,
the following is obtained.

\begin{theorem}
For non-negative integers $v, h$
and an integer $a$,
let $\mathbb{Y}_{v+1, h+1, a}$ be 
the set of the shortest lattice paths
from $(0,0)$ to $(h+1,v+1)$
with the number of cells on the upper-left side
of the path is equivalent to $a$ modulo $v+h+2$.

For any path $\mathbf{y} \in \mathbb{Y}_{v+1, h+1, a}$,
$\mathrm{dS}_{PATH}(\mathbf{y})$ denotes
the set of paths in $\mathbb{Y}_{v,h}$
obtained by 
\begin{itemize}
\item
If $\mathbf{y}$ begins with a horizontal step $y_1$,
delete $y_1$ and a vertical step $y_j$.
\item
If $\mathbf{y}$ begins with a vertical step $y_1$,
delete $y_1$ and a horizontal step $y_j$.
\end{itemize}

Then $\{ \mathbf{dS}_{PATH}(\mathbf{y}) \mid \mathbf{y} \in \mathbb{Y}_{v+1,h+1,a} \}$
is a partition of $\mathbb{Y}_{v, h}$.
\end{theorem}

\begin{example}
For a case with $h=3, v=2, a=1$,
the number of cells for a path
must be $1 \pmod{7(=3+2+2)}$.
There are five such paths 
$0010111$, $0111100$, $1011010$, $1100110$ and $1101001$ with 1 or 8 cells.
(See Figure \ref{figure:expl_PathDel}.)

By the deletion for the five paths,
\begin{align*}
\mathbf{dS}_{PATH}(0010111) &= \{ 00111, 01011 \},\\
\mathbf{dS}_{PATH}(0111100) &= \{ 11100 \},\\
\mathbf{dS}_{PATH}(1011010) &= \{ 11010, 01110, 01101 \},\\
\mathbf{dS}_{PATH}(1100110) &= \{ 10110, 10011 \},\\
\mathbf{dS}_{PATH}(1101001) &= \{ 11001, 10101 \}.
\end{align*}
Hence $\{ \mathbf{dS}(\mathbf{y}) \mid \mathbf{y} \in \mathbb{Y}_{3+1,2+1,1} \}$
covers the all $\binom{3+2}{3}=10$ paths in $\mathbb{Y}_{3,2}$ without overlapping.

\begin{figure}[htbp]
\begin{center}
\includegraphics[width=10cm, bb=0 0 312 195]{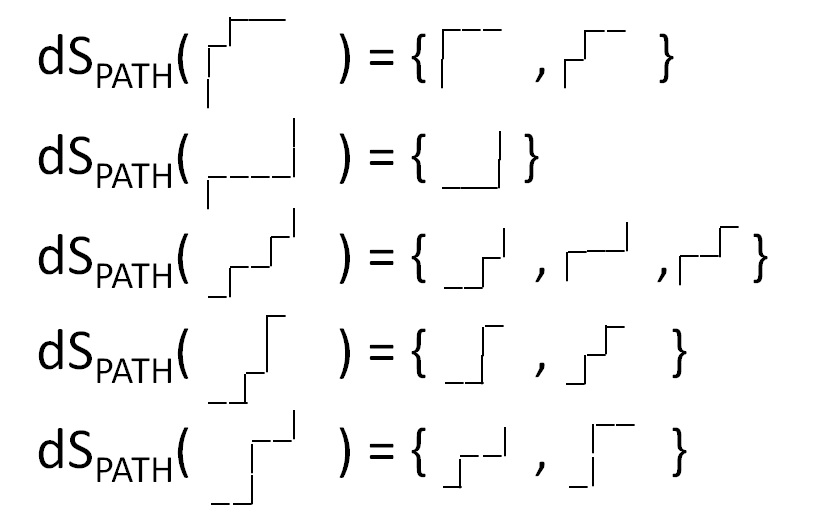} 
\caption{Perfectness for $Y_{3,2,1}$}
\label{figure:expl_PathDel}
\end{center}
\end{figure}
\end{example}

\section{Balanced Adjacent Insertions (BAI) and Deletions (BAD)}\label{sec:BAIBAD}
A main goal of this section is to show the following:
\begin{theorem}\label{thm:BADsPerfect}
For any non-negative integers $v$ and $h$ and any integer $a$,
$$\mathbb{B}_{v+1, h+1, a} := \{ \sigma( \mathbf{x} ) \mid \mathbf{x} \in \mathbb{Y}_{v+1, h+1, a} \}$$
is perfect for ``balanced deletions,''
where $\sigma$ is an azby permutation\footnote{$\sigma(ab \dots yz) = azby \dots$} defined as
$$\sigma( x_1 x_2 x_3 \dots x_{n-1} x_n )
:= x_1 x_n x_2 x_{n-1} x_3 \dots$$
for any positive integer $n$
and balanced adjacent deletions (BADs) are deletion for $01$ or $10$.
\end{theorem}

\begin{example}
Since $\mathbb{Y}_{3,3,1} = \{ 001011,\allowbreak 110010,\allowbreak 101100 \},$
a set $\mathbb{B}_{3,3,1}$ is  $\{010110,\allowbreak 101100,\allowbreak 100011 \}$
By the definition of BADs,
$\mathrm{dS}( 010110 ) = \{ 0110, 0101 \},\allowbreak
\mathrm{dS}( 101100 ) = \{ 1100, 1010 \},\allowbreak
\mathrm{dS}( 100011 ) = \{0011, 1001 \}$.
Hence any pair of deletion spheres has no common elements and $\mathrm{dS}(010110) \cup \mathrm{dS}( 101100 ) \cup \mathrm{dS}( 100011 ) = \mathbb{Y}_{2,2 }$.
\end{example}

To show Theorem \ref{thm:BADsPerfect},
we continue to study $\mathbb{Y}_{v,h}$ (or $\mathbb{Y}_{v+1, h+1}$) and generalized insertions but with
a different reduced expression for the reflection related to the highest coroot.

Our new choice of a reduced expression in $W( A_{v+1, h+1} )$ is
$$
s_{-v} s_{h} s_{-(v-1)} s_{h-1} s_{-(v-2)} s_{h-2} \dots s_{-(v-2)} s_{h-1} s_{-(v-1)} s_{h} s_{-v}.
$$
For example in a case $v=2, h=3$,
the reduced expression is
$$
s_{-2} s_{3} s_{-1} s_{2} s_{0} s_{1} s_{0} s_{2} s_{-1} s_{3} s_{-2}.
$$
As another example in a case $v=2, h=4$,
the reduced expression is 
$$
s_{-2} s_{4} s_{-1} s_{3} s_{0} s_{2} s_{1} s_{2} s_{0} s_{3} s_{-1} s_{4} s_{-2}.
$$

More explicitly, operators $\mathcal{H}^{(v,h)}_i$ ($0 \le i \le 2(v+h)+1$) are defined
as follows:
\begin{eqnarray*}
\mathcal{H}^{(v,h)}_0 &:=& \mathrm{id},\\
\mathcal{H}^{(v,h)}_{2i+1} &:=& s_{-v+i} \mathcal{H}^{(v,h)}_{2i} \;\; (0 \le 2i \le v+h),\\
\mathcal{H}^{(v,h)}_{2i} &:=& s_{h-i} \mathcal{H}^{(v,h)}_{2i-1} \;\; (0 \le 2i-1 \le v+h),\\
\mathcal{H}^{(v,h)}_{2(v+h)+1-2i} &:=& s_{-v + i} \mathcal{H}^{(v,h)}_{2(v+h)-2i} \;\; (0 \le 2i \le v+h),\\
\mathcal{H}^{(v,h)}_{2(v+h)+1-(2i+1)} &:=& s_{h-i} \mathcal{H}^{(v,h)}_{2(v+h)+1-(2i+2)} \;\; (0 \le 2i + 1 \le v+h),
\end{eqnarray*}

Similar to the previous section, 
we define operators $H^{(v,h)}_i := \mathcal{H}^{(v,h)}_i \circ \kappa^{(v,h)}$
and observe their action on $\mathbb{Y}_{v,h}$.
\begin{align*}
H^{(v,h)}_0 (\mathbf{y}) &= 0 y_{-v} y_{-(v-1)} \dots y_{h-1} y_{h} 1,\\
H^{(v,h)}_1 (\mathbf{y}) &= y_{-v} 0 y_{-(v-1)} \dots y_{h-1} y_{h} 1,\\
H^{(v,h)}_2 (\mathbf{y}) &= y_{-v} 0 y_{-(v-1)} \dots y_{h-1} 1 y_{h},\\
H^{(v,h)}_3 (\mathbf{y}) &= y_{-v} y_{-(v-1)} 0 \dots y_{h-1} 1 y_{h},\\
                         & \vdots\\
H^{(v,h)}_{2(v+h)-1} (\mathbf{y}) &= y_{-v} 1 y_{-(v-1)} \dots y_{h-1} 0 y_{h},\\
H^{(v,h)}_{2(v+h)}   (\mathbf{y}) &= y_{-v} 1 y_{-(v-1)} \dots y_{h-1} y_{h} 0,\\
H^{(v,h)}_{2(v+h)+1} (\mathbf{y}) &= 1 y_{-v} y_{-(v-1)} \dots y_{h-1} y_{h} 0.\\
\end{align*}
Hence by similar argument to prove Theorem \ref{thm:I1I2forK},
we can show the following:
\begin{theorem}\label{thm:I1I2forH}
The operations $H^{(v,h)}_i$ satisfy the conditions (I1) and (I2)
for
$
X := \mathbb{Y}_{v,h},
Y := \mathbb{Y}_{v+1, h+1},
r := 2(v+h)+1,
f := \mathrm{inv},
$ and $S := v+h+2$
with notation in Definition \ref{def:generalizedInsertions} (Generalized Insertions).
In other words,
$\mathbf{H}^{(v,h)} := \{ H^{(v,h)}_j \mid 0 \le j \le 2(v+h) + 1 \}$
is a set of generalized insertions.
\end{theorem}

\begin{remark}
Readers may find that any reduced expression of
the reflection for the highest coroot provides us with
generalized insertions.
Once we fix a reduced expression,
observation of the action on bit-sequence
helps us to prove (I1) and (I2).
For preparation of observation,
a characterization of reduced expressions is required.
Since such a characterization is far from the main aim of this paper,
the paper focuses on only two reduced expressions.
\end{remark}

As opposite operations to BADs,
we call two bits 01 or 10 insertions balanced adjacent insertions (BAIs).
Composition of the azby permutation $\sigma$ will clear
a connection between the generalized insertions $H^{(v,h)}_i$ and BAIs.
\begin{align*}
\sigma \circ H^{(v,h)}_0 (\mathbf{y}) &= 0 1 y_{-v} y_{h} y_{-(v-1)} y_{h-1} \dots,\\
\sigma \circ H^{(v,h)}_{2(v+h)+1} (\mathbf{y}) &= 1 0 y_{-v} y_{h} y_{-(v-1)} y_{h-1} \dots.\\
\sigma \circ H^{(v,h)}_1 (\mathbf{y}) &= y_{-v} 1 0 y_{h} y_{-(v-1)} y_{h-1} \dots,\\
\sigma \circ H^{(v,h)}_{2(v+h)}   (\mathbf{y}) &= y_{-v} 0 1 y_{h} y_{-(v-1)} y_{h-1} \dots,\\
\sigma \circ H^{(v,h)}_2 (\mathbf{y}) &= y_{-v} h_{h} 0 1 y_{-(v-1)} y_{h-1} \dots,\\
\sigma \circ H^{(v,h)}_{2(v+h)-1}   (\mathbf{y}) &= y_{-v} y_{h} 1 0 y_{-(v-1)} y_{h-1} \dots,\\
\sigma \circ H^{(v,h)}_3 (\mathbf{y}) &= y_{-v} y_{h} y_{-(v-1)} 1 0 y_{h-1} \dots,\\
\sigma \circ H^{(v,h)}_{2(v+h)-2} (\mathbf{y}) &= y_{-v} y_{h} y_{-(v-1)} 0 1 y_{h-1} \dots,\\
                         & \vdots\\
\end{align*}

Therefore
Theorem \ref{thm:BADsPerfect} is obtained.

\section{Similarity of Cardinalities}\label{sec:SimCard}
This paper discussed
insertion, deletion and perfectness related to minuscule elements
of type $B$ and $A$,
in particular,
Levenshtein codes $\mathbb{L}_{n,a}$, lattice paths $\mathbb{Y}_{v,h,a}$
and sets $\mathbb{B}_{v, h, a}$.
Here a similarity among the cardinalities of these sets are shown.

In this section, $\mu$ denotes a m\"obius function,
$\phi$ the Euler's totient function,
$(d,a)$ the greatest common divisor of $d$ and $a$.

A formula for the cardinality of $\mathbb{L}_{n,a}$ has been obtained by Ginzburg
and independently by Stanley and Yoder and by Sloane.
\begin{fact}
$$
\# \mathbb{L}_{n, a},
= \frac{1}{2 (n+1)} \sum_{ d : \text{ odd }, d | n+1} \mu( \frac{d}{(d, a)}) \frac{ \phi( d ) }{ \phi( \frac{d}{(d, a)} )} 2^{\frac{n+1}{d} }.
$$
(See \cite{ginzburg1967certain,stanley1972study,sloane2000single}.)
\end{fact}

On the other hand, a similar formula for $\# \mathbb{B}_{v, h, a}$ and $\# \mathbb{Y}_{v, h, a}$ to $\# \mathbb{L}_{n, a}$ is obtained
as follows.
\begin{theorem}
$$
\# \mathbb{B}_{v, h, a}
=
\# \mathbb{Y}_{v, h, a}
= \frac{1}{v + h} \sum_{d | (v + h, v)} \mu(  \frac{d}{(d, a)} ) \dfrac{\phi( d )}{\phi(  \frac{d}{(d, a)} )} \binom{(v + h)/d}{v/d} 
$$
\end{theorem}
\begin{proof}
The equality between $\# \mathbb{B}_{v, h, a}$ and $\# \mathbb{Y}_{v, h, a}$
is obvious since the definition of $\mathbb{B}_{v, h, a}$ is
$$
\mathbb{B}_{v, h, a} := \{ \sigma \mathbf{x} \mid \mathbf{x} \in \mathbb{Y}_{v, h, a} \}.
$$

Our strategy to show the second equality is similar to \cite{stanley1972study}.
Set a polynomial $f$ with a variable $q$ as
$$f(q) := \sum_{0 \le a < h + v} \# \mathbb{Y}_{v, h, a} q^a.$$
By the definition of $\mathbb{Y}_{v,h,a}$, the following holds
\begin{align*}
f(q) & = \sum_{\mathbf{x} \in \mathbb{Y}_{v, h}} q^{  \mathrm{inv} ( \mathbf{x} ) \pmod{ v + h}  } \\
     & \equiv \sum_{\mathbf{x} \in \mathbb{Y}_{v, h}} q^{  \mathrm{inv} ( \mathbf{x} ) } \pmod{ q^{v+h} - 1}\\
     & = \begin{bmatrix} v+h \\ h \end{bmatrix},
\end{align*}
where $\begin{bmatrix} v+h \\ h \end{bmatrix}$ is a $q$-binomial.
Therefore 
for a $(v+h)$th primitive root $\zeta$ of $1$
and an integer $j$,
$$
f( \zeta^j ) = 
\begin{cases}
 \binom{ (v + h)/(j, v+h) }{ h / (j, v+h)  } & \text{if $(j, v+h) | (h, v)$,}\\
 0                                        & \text{otherwise.}
\end{cases}
$$

Set a polynomial $H_d(q)$ with $\mathrm{deg} H_d (q) < v + h$ as
$$
H_d (q) := \frac{1}{v+h}
   \sum_{0 \le i < v+h} \gamma(d,i) q^{i},
$$
where $\gamma(d,i) = \sum_{0 \le j < v+h, (j,v+h) = d} \zeta^{ji}$.
It is easy to rewrite as
\begin{align*}
H_d (q)
 &= \frac{1}{v+h} \sum_{0 \le j < v+h, (j,v+h) = d} \sum_{0 \le j < v+h} (\zeta^j q)^i\\
 &= \frac{1}{v+h} \sum_{0 \le j < v+h, (j,v+h) = d} \zeta^{j}
                 \frac{q^{v+h} - 1}{ q - \zeta^j }.
\end{align*}
Therefore for any integer $j$,
$$H_d ( \zeta^j ) = 
\begin{cases}
1 & \text{if $(j, v+h) = d$,}\\
0 & \text{otherwise.}
\end{cases}
$$

It implies that
\begin{align*}
f(q) &= \sum_{d | (v,h)} \binom{(h+v)/d}{v/d} H_d (q) \\
     &= \frac{1}{v+h} \sum_{0 \le a < h+v} \left( \sum_{d | (v,h) } \gamma(d,a) \binom{ (h+v)/d }{ v/d }  \right) q^a .
\end{align*}
By the same argument in \cite{stanley1972study},
$\gamma(d, a) = \mu( \dfrac{d}{(d,a)} ) \dfrac{ \phi(d) }{ \phi( \dfrac{d}{(d,a)} )}$.
It concludes this proof.
\end{proof}

\section{Similarity of Insertion Sphere}\label{sec:SimInsSphere}
This section discusses the size of insertion sphere with
three insertions that have already appeared in this paper.

For a non-negative integer $t \ge 0$,
 a binary sequence $\mathbf{x} \in \{ 0, 1 \}^n$,
let us define the (standard) insertion sphere of degree $t$ as
$$
\mathrm{iS}^{(t)} ( \mathbf{x} ) :=
\begin{cases}
 \{ \mathbf{x} \} & \text{if } t = 0,\\
 \bigcup_{ \mathbf{y}  \in \mathrm{iS}^{(t-1)}( \mathbf{x} ) } \mathrm{iS}( \mathbf{y} ) & \text{otherwise}.
\end{cases}
$$
If we replace standard insertions with path insertions (resp. BAIs), $\mathrm{iS}^{(t)}_{PATH} (\mathbf{x})$
(resp. $\mathrm{iS}^{(t)}_{BA} (\mathbf{x})$)
denotes its insertion sphere.

The following implies that
the size of an insertion sphere for $\mathbf{x} \in \{0,1\}^n$ depends on only its length $n$
and its degree $t$ not but entries of $\mathbf{x}$.
\begin{fact}[Proposition 2 in \cite{d2006reverse}]\label{Thm:insCard}
For any integers $n, t \ge 0$ and any binary sequence $\mathbf{y} \in \{0, 1\}^n$,
$$
\# \mathrm{iS}^{(t)} (\mathbf{y}) = \sum_{0 \le i \le t} \binom{n+t}{i}.
$$
\end{fact}

\begin{remark}
To the best of the author's knowledge,
it is said that 
Theorem \ref{Thm:insCard} was firstly proven in a paper \cite{levenshtein1974elements}.
Since the author could not acquire the paper,
another paper \cite{d2006reverse} is cited here.
In \cite{d2006reverse},
Theorem \ref{Thm:insCard} is proven as a more general statement:
a sequence is not restricted to a binary sequence but a sequence over any finite set.
\end{remark}

One of questions is what happens to the fact 
if $\mathrm{iS}^{(t)}_{PATH} (\mathbf{x})$ and $\mathrm{iS}^{(t)}_{BAI} (\mathbf{x})$ 
are considered.
Answers are Theorems \ref{thm:SizeISpath} and \ref{thm:SizeISBAI}.

\begin{theorem}\label{thm:SizeISpath}
For any non-negative integers $n$ and $t$,
and $\mathbf{x} \in \{0,1\}^n$,
\begin{itemize}
\item $\# \mathrm{iS}^{(0)}_{PATH} (\mathbf{x}) = 1$,
\item $\# \mathrm{iS}^{(1)}_{PATH} (\mathbf{x}) = n + 2$,
\item $\# \mathrm{iS}^{(2)}_{PATH} (00) \neq \mathrm{iS}^{(2)}_{PATH} (01)$.
\end{itemize}
\end{theorem}
\begin{proof}
For a case $t=0$, it follows from
the definition of $\# \mathrm{iS}^{(0)}_{PATH}$.
$$\# \mathrm{iS}^{(0)}_{PATH} (\mathbf{x}) = \# \{ \mathbf{x} \} = 1.$$

Next, for a case $t=1$, 
\begin{align*}
& \# \mathrm{iS}^{(1)}_{PATH} (\mathbf{x}) \\
&= \#\{ \mathbf{y} \mid \mathbf{y} \text{ is obtained by single path insertion to }  \mathbf{x} \}\\
&= \#\{ \mathbf{y} \mid y_{-v} = 0, \mathbf{y} \text{ is obtained by single path insertion to }  \mathbf{x} \}\\
&\;\;\;\; + \#\{ \mathbf{y} \mid y_{-v} = 1, \mathbf{y} \text{ is obtained by single path insertion to }  \mathbf{x} \}\\
&= \#\{ 0 \mathbf{z} \mid \mathbf{z} \text{ is obtained by single $1$ insertion to }   \mathbf{x} \}\\
&\;\;\;\; + \#\{ 1 \mathbf{z} \mid \mathbf{z} \text{ is obtained by single $0$ insertion to }   \mathbf{x} \}\\
&= \#\{ \mathbf{z} \mid \mathbf{z} \text{ is obtained by single $1$ insertion to }   \mathbf{x} \}\\
&\;\;\;\; + \#\{ \mathbf{z} \mid \mathbf{z} \text{ is obtained by single $0$ insertion to }   \mathbf{x} \}\\
&= \#\{ \mathbf{z} \mid \mathbf{z} \text{ is obtained by single standard insertion to }  \mathbf{x} \}\\
&= n+2.
\end{align*}
The last equality follows from Fact \ref{Thm:insCard}.

At the last, for a case $t=2$,
$$
\# \mathrm{iS}^{(2)}_{PATH} (00) = 15 \neq 16 = \mathrm{iS}^{(2)}_{PATH} (01).
$$
Furthermore,
$
\mathrm{iS}^{(2)}_{PATH} (00)
=
\{110000, 101000, 100100, 100010, 100001, \allowbreak 011000, \allowbreak 010100, \allowbreak
 010010,\allowbreak 010001,\allowbreak 001100, 001010, 001001, 000110, 000101, 000011 \},
$
and
$
\mathrm{iS}^{(2)}_{PATH} (01)
=
\{
010101, 001101, 001011,011001, 011010, 010110,\allowbreak 010011,\allowbreak 000111,\allowbreak
100101,\allowbreak 101001,\allowbreak 101010,
100011, 100110, 110001, 110010, 110100
\}.
$
\end{proof}

\begin{theorem}\label{thm:SizeISBAI}
For any non-negative integers $n$ and $t$,
and $\mathbf{x} \in \{0,1\}^n$,
$$\# \mathrm{iS}^{(t)}_{BA} (\mathbf{x}) = \binom{n + 2t}{t}.$$
\end{theorem}

For proving Theorem \ref{thm:SizeISBAI},
Lemmas \ref{lemma:YnttinSp}, \ref{lem:DyckPathTransform}, and \ref{lem:DycPartition} are given.
\begin{lemma}\label{lemma:YnttinSp}
For any positive integers $n, t$ and
any $\mathbf{y} \in \mathbb{Y}_{n + t, t}$,
there exists BAIs $H_1, H_2, \dots, H_t$ such that
$$
\mathbf{y}
=
H_t \circ \dots \circ H_2 \circ H_1 (\mathbf{0}^n),
$$
where $\mathbf{0}^n = 00 \dots 0 \in \{ 0,1 \}^n$.
\end{lemma}
\begin{proof}
Remark that 
for any $\mathbf{y} \in \mathbb{Y}_{n+t, t}$,
$01$ or $10$ is a subword of $\mathbf{y}$.
In other words,
there exists a BAD $D_t$ such that
$D_t (\mathbf{y}) \in \mathbb{Y}_{n+t-1, t-1}$.
Therefore by repeating this remark,
there exist BADs $D_1, D_2, \dots, D_t$ such that
$D_1 \circ D_2 \circ \dots \circ D_t (\mathbf{y}) \in \mathbb{Y}_{n, 0}$,
i.e.,
$$D_1 \circ D_2 \circ \dots \circ D_t (\mathbf{y}) = \mathbf{0}^n.$$
By using the condition (D1) for generalized deletions repeatedly,
there exists BAIs $H_t, \dots, H_2, H_1$ such that
$\mathbf{y} = H_t \circ \dots \circ H_2 \circ H_1 (\mathbf{0}^n)$.
\end{proof}

From the proof above, the following is obtained.
\begin{corollary}\label{cor:iS0nEqY}
For any non-negative integers $n$ and $t$,
$$\mathrm{iS}^{(t)}( \mathbf{0}^n ) = \mathbb{Y}_{n+t, t}.$$
\end{corollary}

For a non-negative integer $c$
and a bit $y$,
let $\mathrm{DYC}( c, y )$
denote the set of bit-sequences of length $2c$
with 
\begin{itemize}
\item the first bit is equal to $y$,
\item for any left subword $x_1 x_2 \dots x_i$ ($1 \le i \le 2c$),
$\# \{ 1 \le j \le i \mid x_j = y \} \ge \# \{ 1 \le j \le i \mid x_j \neq y \}$.
\end{itemize}
An element of $\mathrm{DYC}(c, y)$ is regarded as a Dyck path.
It is known that the cardinality of $\mathrm{DYC}(c, y)$ is a $c$th Catalan number.

\begin{lemma}\label{lem:DyckPathTransform}
For any $\mathbf{l} \in \mathrm{iS}^{(t)}( \text{null} )$
and any bit $y$,
there exist $s, \mathbf{p}$ and $\mathbf{l}'$ such that
$s$ is a non-negative integer,
$\mathbf{p} \in \mathrm{DYC}( s, \bar{y} )$,
$\mathbf{l}' \in \mathrm{iS}^{(t-s)} ( \text{null} )$,
and
$$ \mathbf{l} y = \mathbf{p} y \mathbf{l}',
$$
where $\text{null}$ is the null word,
i.e. the unique element of $\{0,1\}^0$,
and $\bar{y}$ is the flipped bit for $y$, i.e. $\bar{0} = 1$ and $\bar{1} = 0$.
\end{lemma}
\begin{proof}
For any bit sequence $\mathbf{x}$, let $h(\mathbf{x})$ denote
$$
h(\mathbf{x}) := | \mathbf{x} | - 2 \mathrm{wt}( \mathbf{x} ),
$$
where $| \mathbf{x} |$ is the length of $\mathbf{x}$.
In other words,
$$h( \mathbf{x} ) = \text{the number of $0$} - \text{the number of $1$ of $\mathbf{x}$}.$$

Since $h(\mathbf{x}' \mathbf{x}'') = h(\mathbf{x}') + h(\mathbf{x}'')$
for any bit sequences $\mathbf{x}'$ and $\mathbf{x}''$,
we can claim that
\begin{itemize}
\item $h(\mathbf{x} 0) = h(\mathbf{x}) + 1$,
\item $h(\mathbf{x} 1) = h(\mathbf{x}) - 1$
\item $h(\mathrm{l} y) = h(\mathrm{l}) + h(y) = h(y)$.
\end{itemize}

Here we prove the statement for a case $y=1$.
The other case $y=0$ is proven in a similar argument.

Let $\mathbf{p}'$ be a left subword of $\mathbf{l} y$
such that
$$
h( \mathbf{p}' ) = h(\mathbf{p} y) = h(\mathbf{p}) + h(y) = h(y) (=-1)
$$
and
$\mathbf{p}'$ ends with $1$.
Such $\mathbf{p}'$ exists since $\mathbf{l} y$ is an instance.
We chose $\mathbf{p}'$ as the shortest left subwords.
Note that $\mathbf{p}'$ is written as $\mathbf{p} y$ for some bit sequence $\mathbf{p}$.
Since $h(y) = h(\mathbf{p}') = h(\mathbf{p} y) = h(\mathbf{p}) + h(y)$,
$h(\mathbf{p}) = 0$ holds.
On the other hand, $h( \mathbf{q} ) \ge 0$ holds for any left subword $\mathbf{q}$ of $\mathbf{p}$.
If not, it contradicts the choice of $\mathbf{p}$.
Note that $h( \mathbf{q} ) \ge 0$ is equivalent to
the number of $0$ is greater than or equal to the number of $1$ in $\mathbf{q}$,
In other words, $\mathbf{p} \in \mathrm{DYC}( s, \bar{y} )$,
where $s := | \mathbf{p} | / 2$.

Set $\mathbf{l}'$ as the right subword of $\mathbf{l} y$ that is obtained
by deleting $\mathbf{p} y$.
Then the length $| \mathbf{l}' |$ is 
$$
| \mathbf{l}' |
=
| \mathbf{l} y  | - | \mathbf{p} y| = (2t+1) - (2s+1) = 2(t-s).
$$
Since $h(\mathbf{l}) = h(\mathbf{l} y) - h(\mathbf{p} y)
= h(y) - h(y) = 0$,
$\mathbf{l}' \in \mathbb{Y}_{t-s, t-s} = \mathrm{iS}^{(t-s)} (\mathrm{null})$.
\end{proof}

\begin{lemma}\label{lem:DycPartition}
$\mathbf{y} = y_1 y_2 \dots y_n \in \{ 0, 1 \}^n$.

For any $\mathbf{x} \in \mathrm{iS}^{(t)}( y )$,
there exists 
$c_1, c_2, \dots, c_{n+1} \ge 0$ with $c_1 + c_2 + \dots + c_{n+1} = t$
such that
$$
\mathbf{x} = \mathbf{p}_1 y_1 \mathbf{p}_2 y_2
\dots \mathbf{p}_t y_t \mathbf{l}_{t+1},$$
where
$\bar{y}_i := 1 - y_i$, 
$\mathbf{p}_i \in \mathrm{DYC}( c_i, \bar{y_i} )$ (for $1 \le i \le n$),
and $\mathbf{l}_{t+1} \in \mathbb{Y}_{c_{n+1}, c_{n+1} }$.
\end{lemma}
\begin{proof}
Since $\mathbf{x}$ is obtained by $t$-BAIs to $\mathbf{y}$,
$\mathbf{x}$ is written as
$$
\mathbf{x} = \mathbf{l}_1 y_1 \mathbf{l}_2 y_2 \dots \mathbf{l}_n y_n \mathbf{l}_{n+1}. 
$$
By applying Lemma \ref{lem:DyckPathTransform} repeatedly,
\begin{align*}
\mathbf{x}
 &= \mathbf{l}_1 y_1 \mathbf{l}_2 y_2 \dots \mathbf{l}_n y_n \mathbf{l}_{n+1}\\
 &= \mathbf{p}_1 y_1 \mathbf{l}'_1 \mathbf{l}_2 y_2 \dots \mathbf{l}_n y_n \mathbf{l}_{n+1}\\
 &= \mathbf{p}_1 y_1 \mathbf{p}_2 y_2 \mathbf{l}'_2 \dots \mathbf{l}_n y_n \mathbf{l}_{n+1}\\
 & \vdots\\
 &= \mathbf{p}_1 y_1 \mathbf{p}_2 y_2 \dots \mathbf{l}'_{n-1} \mathbf{l}_n y_n \mathbf{l}_{n+1}\\ 
 &= \mathbf{p}_1 y_1 \mathbf{p}_2 y_2 \dots \mathbf{p}_n y_n \mathbf{l}'_{n} \mathbf{l}_{n+1}\\
 &= \mathbf{p}_1 y_1 \mathbf{p}_2 y_2 \dots \mathbf{p}_n y_n \mathbf{l}'_{n+1}.
\end{align*}
The last equation holds by setting $\mathbf{l}'_{n+1} := \mathbf{l}'_{n} \mathbf{l}_{n+1}$.

By defining $c_i := | \mathbf{p}_i | / 2$,
the statement is proven.
\end{proof}

Lemma \ref{lem:DycPartition} guarantees existence of
a sequence of non-negative integers $c_1, c_2, \dots, c_{n+1}$ for any element
of $\mathrm{iS}^{(t)}( \mathbf{y} )$
with the property in the Lemma.
By considering lexicographic ordering,
there uniquely exists the minimum $c_1, c_2, \dots, c_{n+1}$.
Let $\mathrm{DYC}( \mathbf{y} ; c_1, c_2, \dots, c_{n+1} )$
be the set of such elements.

\begin{proof}[Proof for Theorem \ref{thm:SizeISBAI}]
By Corollary \ref{cor:iS0nEqY},
$\# \mathrm{iS}^{(t)} ( \mathbf{0}^n ) = \# \mathbb{Y}_{n+t, t} = \binom{n+2t}{t}$.

Next, we prove $\# \mathrm{iS}^{(t)} ( \mathbf{0}^n ) = \# \mathrm{iS}^{(t)} ( \mathbf{y} )$.

It is clear that $\mathrm{DYC}( \mathbf{y} ; c_1, c_2, \dots, c_{n+1} ) \subset \mathrm{iS}^{(t)}_{BA} ( \mathbf{y} )$.

Lemma implies that
$\mathrm{DYC}( \mathbf{y} ; c_1, c_2, \dots, c_{n+1})$ provides us with
the following partition:
$$
\mathrm{iS}^{(t)}_{BA} ( \mathbf{y} ) =
\sqcup_{c_1, c_2, \dots, c_{n+1}} \mathrm{DYC}( \mathbf{y} ; c_1, c_2, \dots, c_{n+1}).
$$
It is easy to check that
$
0 x_2 \dots x_{2c}  \in \mathrm{DYC}(c, 0)
$
if and only if
$1 \bar{x}_2 \dots \bar{x}_{2c} \in \mathrm{DYC}(c, 1)$.
This implies that $\# \mathrm{DYC}( \mathbf{y} ; c_1, c_2, \dots, c_{n+1} )
= \# \mathrm{DYC}( \mathbf{0}^n ; c_1, c_2, \dots, c_{n+1} )$ for any $\mathbf{y} \in \{0,1\}^n$.
Hence the cardinality is independent on the choice of $\mathbf{y}$.

Furthermore the cardinality is
\begin{align*}
\# \mathrm{iS}^{(t)}_{BA} ( \mathbf{y} )
&= \sum_{c_1, c_2, \dots, c_{n+1}} \# \mathrm{DYC}(y; c_1, c_2, \dots, c_{n+1})\\
&=  \sum_{c_1, c_2, \dots, c_{n+1}} C_{c_1} C_{c_2} \dots C_{c_n } \# \mathbb{Y}_{c_{n+1}, c_{n+1}},
\end{align*}
where $C_i$ is the $i$th Catalan number.
\end{proof}

\begin{corollary}
$$
\sum_{c_1, c_2, \dots, c_{n+1} \ge 0, c_1 + c_2 + \dots + c_{n+1} = t}
   (c_{n+1} + 1) C_{c_1} C_{c_2} \dots C_{c_{n+1}}
= \binom{n + 2t }{ t }.
$$
\end{corollary}
\begin{proof}
It immediately follows from 
the proof for Theorem \ref{thm:SizeISBAI}
and
$$C_{c_{n+1}} = \frac{1}{c_{n+1}+1} \binom{ 2 c_{n+1} }{ c_{n+1} }
\text{ and }
\# \mathbf{Y}_{c_{n+1}, c_{n+1} } = \binom{ 2 c_{n+1} }{ c_{n+1} }.
$$
\end{proof}

\section{Conclusion}
This paper pointed out a connection between standard insertions for bits
and action by right subwords of the reduced expression of a reflection
related to the highest coroot on minuscule elements of type $B$.
Inspired by the connection, generalized insertions, path insertions and balanced
adjacent insertions are defined.
Furthermore, similar combinatorial properties to standard insertions and 
ID codes are obtained.
The author expect that more similar properties will be found to generalized
insertions.
Various interesting properties of standard insertions have been found since 1960's.

The generalized insertions introduced in this paper are related to minuscule
elements only of type $A$.
The author is trying to generalize insertions to other types, for example D.
A general argument that does not depend on a type of Weyl group will be
future work while it seems to be difficult.

\section*{Acknowledgment}
This paper is partially supported by
KAKENHI 16K12391, 18H01435 and 16K06336.
The author would like to thank Dr. Kento Nakada, Dr. Taro Sakurai and Prof. Richard Green for valuable comments.

\bibliography{reference}

\begin{thebibliography}{10}

\bibitem{biagioli2013fully}
Riccardo Biagioli, Fr{\'e}d{\'e}ric Jouhet, and Philippe Nadeau.
\newblock Fully commutative elements and lattice walks.
\newblock In {\em 25th International Conference on Formal Power Series and
  Algebraic Combinatorics (FPSAC 2013)}, pages 145--156. Discrete Mathematics
  and Theoretical Computer Science, 2013.

\bibitem{bibak2018weight}
Khodakhast Bibak and Olgica Milenkovic.
\newblock Weight enumerators of some classes of deletion correcting codes.
\newblock In {\em 2018 IEEE International Symposium on Information Theory
  (ISIT)}, pages 431--435. IEEE, 2018.

\bibitem{brill2000improved}
Eric Brill and Robert~C Moore.
\newblock An improved error model for noisy channel spelling correction.
\newblock In {\em Proceedings of the 38th Annual Meeting on Association for
  Computational Linguistics}, pages 286--293. Association for Computational
  Linguistics, 2000.

\bibitem{chee2017codes}
Yeow~Meng Chee, Han~Mao Kiah, Alexander Vardy, Eitan Yaakobi, et~al.
\newblock Codes correcting position errors in racetrack memories.
\newblock In {\em Information Theory Workshop (ITW), 2017 IEEE}, pages
  161--165. IEEE, 2017.

\bibitem{d2006reverse}
Arkadii D’yachkov, David Torney, Pavel Vilenkin, and Scott White.
\newblock Reverse--complement similarity codes.
\newblock In {\em General Theory of Information Transfer and Combinatorics},
  pages 814--830. Springer, 2006.

\bibitem{ginzburg1967certain}
BD~Ginzburg.
\newblock A certain number-theoretic function which has an application in
  coding theory.
\newblock {\em Problemy Kibernet}, 19:249--252, 1967.

\bibitem{green2002321}
Richard~M Green.
\newblock On 321-avoiding permutations in affine weyl groups.
\newblock {\em Journal of Algebraic Combinatorics}, 15(3):241--252, 2002.

\bibitem{green2013combinatorics}
Richard~M Green.
\newblock {\em Combinatorics of minuscule representations}, volume 199.
\newblock Cambridge University Press, 2013.

\bibitem{hagiwara2004minuscule}
Manabu Hagiwara.
\newblock Minuscule heaps over dynkin diagrams of type a.
\newblock {\em JOURNAL OF COMBINATORICS}, 11(1):R3, 2004.

\bibitem{hagiwara2017perfect}
Manabu Hagiwara.
\newblock Perfect codes for single balanced adjacent deletions.
\newblock In {\em Information Theory (ISIT), 2017 IEEE International Symposium
  on}, pages 1938--1942. IEEE, 2017.

\bibitem{humphreys1992reflection}
James~E Humphreys.
\newblock {\em Reflection groups and Coxeter groups}, volume~29.
\newblock Cambridge university press, 1992.

\bibitem{jouhet2014long}
Fr{\'e}d{\'e}ric Jouhet and Philippe Nadeau.
\newblock Long fully commutative elements in affine coxeter groups.
\newblock {\em arXiv preprint arXiv:1407.5575}, 2014.

\bibitem{kurtz2001reputer}
Stefan Kurtz, Jomuna~V Choudhuri, Enno Ohlebusch, Chris Schleiermacher, Jens
  Stoye, and Robert Giegerich.
\newblock Reputer: the manifold applications of repeat analysis on a genomic
  scale.
\newblock {\em Nucleic acids research}, 29(22):4633--4642, 2001.

\bibitem{levenshtein1966binary}
Vladimir~I Levenshtein.
\newblock Binary codes capable of correcting deletions, insertions, and
  reversals.
\newblock {\em Soviet physics doklady}, 10(8):707--710, 1966.

\bibitem{levenshtein1974elements}
Vladimir~I Levenshtein.
\newblock Elements of coding theory (in russian).
\newblock {\em Descrete Mathematics and Mathematical Problems of Cybernetics},
  pages 207--305, 1974.

\bibitem{levenshtein1992perfect}
Vladimir~I Levenshtein.
\newblock On perfect codes in deletion and insertion metric.
\newblock {\em Discrete Mathematics and Applications}, 2(3):241--258, 1992.

\bibitem{och2003minimum}
Franz~Josef Och.
\newblock Minimum error rate training in statistical machine translation.
\newblock In {\em Proceedings of the 41st Annual Meeting on Association for
  Computational Linguistics-Volume 1}, pages 160--167. Association for
  Computational Linguistics, 2003.

\bibitem{ossowski2008sequencing}
Stephan Ossowski, Korbinian Schneeberger, Richard~M Clark, Christa Lanz, Norman
  Warthmann, and Detlef Weigel.
\newblock Sequencing of natural strains of arabidopsis thaliana with short
  reads.
\newblock {\em Genome research}, pages gr--080200, 2008.

\bibitem{proctor1999dynkin}
Robert~A Proctor.
\newblock Dynkin diagram classification of $\lambda$-minuscule bruhat lattices
  and of d-complete posets.
\newblock {\em Journal of Algebraic Combinatorics}, 9(1):61--94, 1999.

\bibitem{proctor1999minuscule}
Robert~A Proctor.
\newblock Minuscule elements of weyl groups, the numbers game, and d-complete
  posets.
\newblock {\em Journal of Algebra}, 213(1):272--303, 1999.

\bibitem{sloane2000single}
Neil~JA Sloane.
\newblock On single-deletion-correcting codes.
\newblock {\em Codes and designs}, 10:273--291, 2000.

\bibitem{stanley1972study}
Richard~P Stanley and Michael~F Yoder.
\newblock A study of varshamov codes for asymmetric channels.
\newblock {\em Jet Prop. Lab. Tech. Rep}, pages 32--1526, 1972.

\bibitem{stembridge1996fully}
John~R Stembridge.
\newblock On the fully commutative elements of coxeter groups.
\newblock {\em Journal of Algebraic Combinatorics}, 5(4):353--385, 1996.

\bibitem{stembridge2001minuscule}
John~R Stembridge.
\newblock Minuscule elements of weyl groups.
\newblock {\em Journal of Algebra}, 235(2):722--743, 2001.

\bibitem{vahid2017correcting}
Alireza Vahid, Georgios Mappouras, Daniel~J Sorin, and Robert Calderbank.
\newblock Correcting two deletions and insertions in racetrack memory.
\newblock {\em arXiv preprint arXiv:1701.06478}, 2017.

\bibitem{varshamov1965code}
RR~Varshamov and GM~Tenenholtz.
\newblock A code for correcting a single asymmetric error.
\newblock {\em Automatica i Telemekhanika}, 26(2):288--292, 1965.

\bibitem{xu2005survey}
Rui Xu and Donald Wunsch.
\newblock Survey of clustering algorithms.
\newblock {\em IEEE Transactions on neural networks}, 16(3):645--678, 2005.

\end{thebibliography}

\end{document}